\documentclass[twoside,a4paper,11pt,limits]{amsart}
\usepackage{paper_diening}
\usepackage{graphicx}
\usepackage{color}
\usepackage{dsfont}

\usepackage{amssymb}
\usepackage{amsmath}
\usepackage{amssymb}
\usepackage{amsthm}
\usepackage{bbm}
\usepackage{color}

\usepackage{enumitem}
\usepackage{comment}

\newtheorem{theorem}{Theorem}[section]
\newtheorem{lemma}[theorem]{Lemma}
\newtheorem{corollary}[theorem]{Corollary}

\newtheorem{rem}[theorem]{Remark}

\newtheorem{assumption}[theorem]{Assumption}

\numberwithin{equation}{section}

\newcommand{\RN}{{\setR^N}}
\newcommand{\RNn}{\setR^{n\times N}}

\newcommand{\dt}{\ensuremath{\,{\rm d} t}}

\newcommand{\dx}{\ensuremath{\,{\rm d} x}}
\newcommand{\dy}{\ensuremath{\,{\rm d} y}}
\newcommand{\dz}{\ensuremath{\,{\rm d} z}}
\newcommand{\deps}{\ensuremath{\,{\rm d} \varepsilon}}

\newcommand{\Acal}{{\mathcal{A}}}
\providecommand{\wal}{\ensuremath{w_\lambda}}

\providecommand{\Oal}{\ensuremath{\mathcal{O}_\lambda}}

\providecommand{\Acal}{\mathcal{A}}

\def\esup{\ensuremath{\operatorname*{ess\, sup}}}

\begin{document}

\title[Nonlinear parabolic systems]{Well posedness of nonlinear parabolic systems beyond duality}\thanks{M.~Bul\'{\i}\v{c}ek's work is supported by  the Czech Science Foundation (Grant no. 18-12719S). M.~Bul\'{\i}\v{c}ek and S. Schwarzacher are members of the Ne\v{c}as Center for Mathematical Modeling. J. Burczak was supported by the National Science Centre, Poland (NCN) grant `SONATA'
2016/21/D/ST1/03085. J. Burczak was supported by MNiSW "Mobilnosc Plus" grant
1289/MOB/IV/2015/0.}

\author[M.~Bul\'{\i}\v{c}ek]{Miroslav Bul\'{\i}\v{c}ek} 
\address{Mathematical Institute, Faculty of Mathematics and Physics, Charles University
Sokolovsk\'{a} 83, 186 75 Prague, Czech Republic}
\email{mbul8060@karlin.mff.cuni.cz}

\author[J. Burczak]{Jan Burczak} 
\address{Mathematical Institute, University of Oxford, UK and Institute of Mathematics, Polish Academy of Sciences, \'Sniadeckich 8, 00-656 Warsaw,
Poland}
\email{burczak@maths.ox.ac.uk}

\author[S.~Schwarzacher]{Sebastian Schwarzacher}
\address{Mathematical Institute, Faculty of Mathematics and Physics,  Charles University,
Sokolovsk\'{a} 83, 186 75 Prague, Czech Republic}
\email{schwarz@karlin.mff.cuni.cz}

\begin{abstract}
We develop a methodology for proving well-posedness in optimal regularity spaces for a wide class of nonlinear parabolic initial-boundary value systems, where the standard monotone operator theory fails.
A motivational example of a problem accessible to our technique is the following system
\[
  \partial_tu-\divergence \left( \nu(|\nabla u|) \nabla u \right)= -\divergence f
\]
with a given {strictly} positive bounded function $\nu$, {such that $\lim_{k\to \infty} \nu(k)=\nu_\infty$} and $f \in L^q$ with $q\in (1,\infty)$. The {existence, uniqueness and regularity} results for $q\ge 2$ are by now standard. However, even if a priori estimates  are available, the existence in case $q\in (1,2)$ was essentially missing. We overcome the related crucial difficulty, namely the lack of a standard duality pairing, by resorting to proper weighted spaces and consequently provide existence, uniqueness and optimal regularity in the entire range $q\in (1,\infty)$.

Furthermore, our paper includes several new results that may be of independent interest and serve as the starting point for further analysis of more complicated problems. They include a parabolic Lipschitz approximation method in weighted spaces with fine control of the time derivative and a theory for linear parabolic systems with right hand sides belonging to Muckenhoupt weighted $L^q$ spaces.
\end{abstract}
%


\maketitle

\section{Introduction}
\label{S1}

\noindent
We consider the following system of partial differential equations
\begin{align}
\label{eq:sysA}
\begin{aligned}
  \partial_tu-\divergence A(z; \nabla u) &= -\divergence
  f &&\textrm{ in } Q_T,\\
  u&=\tilde{g}&&\textrm{ on } \partial Q_T,
  \end{aligned}
\end{align}
where $z = (x,t) \in \Omega \times (0,T) \equiv {Q_T}$ with a bounded $\Omega \subset \mathbb{R}^n$, $n\ge 2$ and $\partial Q_T \equiv (\partial \Omega \times (0,T)) \cup  (\Omega \times \{0\})$ is the parabolic boundary. The unknown is the vector $u: {Q_T} \to \mathbb{R}^N$ with $N \in
\mathbb{N}$, whereas the given data set consists of: the forcing $f: {Q_T} \to \mathbb{R}^{n\times N}$, 
the initial-boundary values $\tilde{g}:\partial Q_T\to \setR^N$ and the nonlinear mapping $A: {Q_T} \times\mathbb{R}^{n\times N}\to \mathbb{R}^{n\times N}$.

Our primary interest lies in studying well-posedness of  \eqref{eq:sysA} for right hand sides $f$ that evade the standard theory\footnote{Difficulties arising from irregularity of the datum $\tilde{g}$ are of secondary concern to us and the ones related to roughness of the domain or of the tensor $A$ are even less so.}. Let us explain this in more detail. For the sake of clarity we focus on the case $\tilde{g} \equiv 0$.

In the simplest situation, i.e.,for $A$ being linear, strongly elliptic and sufficiently smooth, one has $f\in L^q(Q_T,\mathbb{R}^{n\times N}) \implies \nabla u \in L^q(Q_T,\mathbb{R}^{n\times N})$ in the whole range $q \in (1, \infty)$,~\cite{Byunatall15}. It is strongly related to second order $L^q$-theory~\cite[Ch. IV]{LadSolUra68}.

The picture becomes much more entangled when we enter the nonlinear realm. The first step is to consider mappings $A$ that are monotone and of linear growth. Naturally via the standard monotone operator theory $f\in L^2(Q_T,\mathbb{R}^{n\times N})$ implies existence of a unique weak solution $u\in L^2(W^{1,2})$ to~\eqref{eq:sysA}. Moreover, for higher integrable $f$'s we can obtain respective regularity of the solution~\cite{AceMin07,Bog14}.

On the other hand, the theory for $f \in L^q(Q_T,\mathbb{R}^{n\times N}), q\in (1,2)$ \emph{ is essentially missing}. Hence, the challenge we have set ourselves reads
\begin{center}
{\em Develop well posedness theory for \eqref{eq:sysA} with $f\in L^q$ in the entire range $q\in (1,\infty)$.}
\end{center}
We believe that our main results, i.e., Theorems \ref{thweak} and \ref{th} settle this matter to a considerable extent, within monotone mappings $A$ of linear growth. 

Let us emphasise that our main contribution is less in deriving a priori estimates, but rather in providing an existence and uniqueness methodology. Indeed, on the one hand, a priori estimates are (formally) rather straightforward within our quite restrictive assumption~\eqref{eq:ass1}. On the other hand, by inspecting the monotone operator theory one realises that it essentially requires the mapping $u\mapsto \divergence A(\cdot,\nabla u)$ to couple via duality. Thus, since in the case $q\in (1,2)$ a-priori estimates in $W^{1,q}$ do not allow for a (standard) duality pairing, they seem too weak to provide existence or uniqueness (in correct, i.e., optimal regularity classes). 
We circumvent this obstacle by utilising a weighted duality pairing.

\subsection{Context and related results}

Our paper continues the line of research initiated recently in \cite{BulDieSch15}, where well-posedness for all $q \in (1, \infty)$ was provided in the elliptic case. The key idea developed there was to replace the non-available $L^q$ duality with the one in a Muckenhoupt-weighted $L^2_\omega$ space, with $\omega:=(Mf)^{q-2}$, where $M$ is the Hardy-Littlewood maximal operator. 
 The linear growth bound on $A$ implies then that $u$ and $\divergence A(\cdot,\nabla u)$ form a duality couple in $L^2_\omega$. The result of \cite{BulDieSch15} was further extented in~\cite{BulBurSch16} for the steady systems covering flows of incompressible fluids and in~\cite{BulSch16} for the to the $p$-Laplacian setting for $q$ smaller but close to $p$. 
%

 The main contribution of this paper is showing that a result analogous to that of \cite{BulDieSch15} holds true for the parabolic system~\eqref{eq:sysA}. 
 Let us briefly explain the main parabolic challenges. Firstly, one has to provide {\em parabolic a-priori estimates in weighted spaces}. In the case studied here they essentially rely on weighted estimates for the linear case, which are new for general Muckenhoupt weights (certain special cases, not applicable for our nonlinear purposes, may be found in \cite{Byunatall15}). Thus we believe that our main linear estimates of Theorem \ref{thm:lin} may be of independent interest. The second challenge, namely {\em exploiting the monotonicity by the weighted duality}, is more pivotal to the whole reasoning. Here the celebrated Lipschitz truncation method~\cite{AceF84,Lew93,KinLew02,DieRuzWol10} comes into play and has to be refined according to the weighted estimates. We built on the parabolic Lipschitz truncation first developed in~\cite{KinLew02}. Our estimates heavily rely on the more recent version of the parabolic Lipschitz truncation constructed in~\cite{DieSchStrVer16}. The technical highlight of the present paper's Lipschitz approximation provided in Theorem \ref{thm:lipseq} represents our new and rather surprising fine control of the distributional time derivative.

Concerning related nonlinear results, let us recall that for the parabolic $p$-Laplacian higher integrability results are available in the framework of { weak solutions}, i.e., within duality pairing, cf.
~\cite{KinLew00,AceMin07,Bog14} and even continuity bounds for the gradient are known~\cite{DiBFri85,DiB93,Mis02,Bur12, Sch14}. 
However, in the framework of very-weak solutions  (beyond duality pairing) estimates are only available for exponents $q$ close to $p$, cf.~\cite{KinLew02}. Moreover, an existence or uniqueness theory for very weak solutions is missing for the parabolic $p$-Laplacian.
 However, in the elliptic case for $q$ close to $p$ existence was shown~\cite{BulSch16}, where the elliptic strategy developed in~\cite{BulDieSch15} was successfully implemented. 

Finally, let us observe that the case of $f\in L^q$, $q \in (1,2)$ goes much beyond the measure-valued right hand sides that have been attracting a considerable attention, compare \cite{BreFri83,BocGal92,BocGal97,BlaMur97,BlaMur01} for some scalar cases. Despite that, for measures of any form uniqueness of solutions was missing. Hence, it seems worth emphasizing that our main result provides in particular existence and uniqueness for systems with measure-valued right-hand sides, compare Corollary~\ref{CMB}.
%


{\em Paper structure:}
In the remaining part of this section, we formulate precisely the assumptions on the data and state our main results, immediately after introducing solely the fundamental notions. Further definitions and certain auxiliary results are gathered in Section \ref{S2}. The subsequent two sections contain our main technical contributions: i.e., a refined parabolic Lipschitz approximation and the linear weighted theory, respectively. Finally, Section \ref{sec:pf} is devoted to the proof of our main results.

\subsection{Assumptions}\label{S:Results}
Throughout the whole paper, we use the standard notation for Lebesgue, Sobolev and Bochner spaces respectively. The conjugate exponent to $q$ is denoted by $q' := q/(q-1)$. For precise definitions of weighted spaces and for the notion of Muckenhoupt weights, we refer to Section~\ref{S2}.

First, let us consider the nonlinear tensor $A$.
\begin{assumption}
\label{ass:A}
Let ${A}: Q_T \times \RNn \to \RNn$ be a Carath\'{e}odory mapping such that for certain positive numbers $c_0, c_1, c_2$ there hold the following linear growth bounds for all $Q\in \RNn$
\begin{align}
\label{ass:lin}
c_0\abs{Q}^2- c_2\leq {A}({z}; Q)\cdot Q, \qquad \abs{{A}({z}; Q)}\leq c_1\abs{Q} +c_2.
\end{align}
Moreover, let $A$ be monotone, i.e., for all $Q,P\in \RNn$
\begin{align}
\label{ass:mon}
0\leq({A}({z}; Q)-{A}({z}; P))\cdot(Q-P).
\end{align}
Moreover, we assume that $A$ is \emph{linear-at-infinity} in the following sense: there exists a mapping $\tilde {A}\in L^{\infty}(Q_T; \RNn\times \RNn)$ and positive constants $\tilde{c}_0$ and $\tilde{c}_1$ such that for all $\eta \in  \mathbb{R}^{n\times N}$ and all $z\in Q_T$, we have
\begin{equation} \label{tildeA}
\tilde{c}_0|\eta|^2\le {\tilde A (z)} \eta \cdot \eta \le \tilde{c}_1|\eta|^2
\end{equation}
and
\begin{align}
\label{eq:ass1}
\lim_{\abs{Q}\to\infty} \esup_{z\in Q_T}\frac{\abs{{A}({z}; Q)- \tilde A (z) Q}}{\abs{Q}}=0.
\end{align}
\end{assumption}
Of course, the motivational example from the abstract
\begin{align}
\label{eq:mot}
A(Q)=\nu(\abs{Q})Q
\end{align}
with a given {strictly} positive bounded function $\nu$, such that $\lim\limits_{k\to \infty} \nu(k)=\nu_\infty$ and $\nu'\geq -1$ falls within the Assumption~\ref{ass:A}. 

Assumption~\ref{ass:A} suffices for the existence theory in optimal regularity spaces. However, for the uniqueness, we shall require a slightly stronger
\begin{assumption}
\label{ass:B}
Let $A$ satisfy Assumption~\ref{ass:A} and in addition it also fulfils
\[
\lim_{\abs{Q}\to\infty}  \esup_{z\in Q_T} \left|\frac{\partial {A} ({z}; Q)}{\partial Q}- {\tilde A (z)} \, \right| = 0.
\]
\end{assumption}
Since in \eqref{eq:mot} admissible choices are
\[
\nu(\abs{Q})=\min\{\nu_\infty,\abs{Q}^{p-2}\}\text{ for }p\in (2,\infty)\text{  or }\nu(\abs{Q})=\max\{\nu_\infty,\abs{Q}^{p-2}\}\text{ for }p\in (1,2),
\]
the paper covers certain approximations of the parabolic $p$-Laplacian. 

\vskip 2mm

Let us now specify our assumptions on the initial and boundary data $\tilde{g}$. Since we are dealing with the solutions that are beyond the framework of duality, which  means that $\tilde{g}$ might not belong to  the natural trace-space with respect to the parabolic operator, namely 
\[
\set{\tilde{g}\in L^2(0,T; W^{1/2,2}(\partial \Omega;\mathbb{R}^N))\,:\,\tilde{g}(0)\in L^2(\Omega;\mathbb{R}^N)},
\] 
we have to proceed carefully while introducing a notion of boundary and initial conditions. In order to simplify our presentation and also to avoid the unwieldy technical tools  (function space related), we rather prescribe the space-time traces as being attained by a certain function $g$, which can be given e.g. by a heat flow. More precisely, we have
\begin{assumption}
\label{ass:C}
Let the initial-boundary data $\tilde{g}$ be of the form $\tilde{g}(0) \in (\mathcal{D}(\Omega;\mathbb{R}^N))^*$ and $\tilde{g}\in L^1(0,T; L^1(\partial \Omega;\mathbb{R}^N))$. Furthermore, we require that  there exist $g\in L^1(0,T; W^{1,1}(\Omega;\mathbb{R}^N))$ and $F\in L^1(Q_T; \RNn)$ such that $g=\tilde{g}$ on $(0,T)\times \partial \Omega$ and that
\begin{equation} \label{bypart}
\int_0^T  g(z) \cdot \partial_t\varphi(z) \dz + \langle \tilde{g}(0), \varphi(0)\rangle_{\mathcal{D}(\Omega)} = \int_Q  F(z) \cdot \nabla \varphi(z) \dz
\end{equation}
is satisfied for all $\varphi\in \mathcal{D}((-\infty,T)\times \Omega;\mathbb{R}^N)$.
\end{assumption}
Notice here that \eqref{bypart} is nothing else than
\begin{equation}\label{represent}
\partial_t g = \divergence F \qquad \textrm{ and } \qquad g(0)=\tilde{g}(0)
\end{equation}
in sense of distribution.

\subsection{Weak and very weak solutions}
At this point, we can define a notion of a weak solution to~\eqref{eq:sysA}. Hence, for $A$ satisfying Assumption~\ref{ass:A}, $\tilde{g}$ satisfying Assumption~\ref{ass:C} and $f\in L^1(Q_T; \RNn)$ we look for $u\in L^1(0,T; W^{1,1}(\Omega; \RN))$ such that $(u-g) \in L^1(0,T; W^{1,1}_0(\Omega; \RN))$ and for all $ \phi\in \mathcal{C}^{1}_0(\Omega \times (-\infty, T);\mathbb{R}^N)$  there holds
\begin{equation}\label{eq:WL}
 \int_{{Q_T}} \left[ (- u (z)+g(z)) \cdot \partial_t \phi (z) + A(z; \nabla u (z)) \cdot \nabla \phi (z)\right] \dz =\int_{Q_T}  (f (z)+F(z)) \cdot \nabla \phi (z) \dz.
\end{equation}
Thanks to the growth assumption on $A$, see Assumption~\ref{ass:A}, and the fact that $\nabla u\in L^1(Q_T;\mathbb{R}^{n\times N})$ all integrals in \eqref{eq:WL} are well defined. 
Please observe that, in view of the representation Assumption~\ref{ass:C}, the identity \eqref{eq:WL} is the distributional formulation of~\eqref{eq:sysA}.


If $u$ and $\divergence A(\cdot,\nabla u)$ are not coupled via a duality, i.e., if $\nabla u\not\in L^2$, then we call $u$ a {\em very weak solution}. If $\nabla u \in L^2$ then we call $u$ a {\em weak solution}. In particular, in the regime of {\em weak solutions} $u-g$ can be used as a test function in~\eqref{eq:WL} but not in the very weak regime.


\subsection{Uniqueness}\label{ssec:uq}
We will call a solution to \eqref{eq:sysA} \emph{unique in the $L^s(0,T; W^{1,s}(\Omega;\RN))$ class} provided the following holds: Take any two solutions $u_1, u_2 \in L^s(0,T; W^{1,s}(\Omega;\RN))$ to \eqref{eq:sysA} with data $(g_1,F_1,f)$ and $(g_2,F_2,f)$ respectively. If $(g_1,F_1)$ and $(g_2,F_2)$ satisfy $(g_1-g_2) \in L^{s}(0,T; W^{1,s}_0(\Omega;\RN))$ and for all $\varphi\in \mathcal{C}^1_0((-\infty,T)\times \Omega;\mathbb{R}^N)$
\begin{equation}\label{uniq:nonlin}
\int_{Q_T} (g_1-g_2)\cdot \partial_t \varphi - (F_1-F_2)\cdot \nabla \varphi \dz=0,
\end{equation}
then $u_1=u_2$ almost everywhere in $Q_T$. 

\subsection{Results}
We are ready to state our main results. 
\begin{theorem}\label{thweak}
Let $q\in (1,\infty)$ be arbitrary and $\Omega \in \mathcal{C}^1$. Assume that  ${A}$ satisfies Assumption~\ref{ass:A}, $\tilde{g}$ satisfies Assumption~\ref{ass:C} with $g\in L^q(0,T;W^{1,q}(\Omega;\mathbb{R}^N))$ and $F\in L^q(Q_T;\mathbb{R}^{n\times N})$ and assume that $f\in L^q (Q_T;\mathbb{R}^{n\times N})$. Then  \eqref{eq:sysA} admits a solution $u$ satisfying~\eqref{eq:WL} such that
\begin{equation}\label{eq:optweak}
\norm{u-g}_{L^q (0,T; W_0^{1,q} (\Omega;\RN))}  \le C \left(1+ \norm{f}_{L^{q} (Q_T;\RNn)}+\norm{\nabla g}_{L^{q} (Q_T;\RNn)}+\norm{F}_{L^{q} (Q_T;\RNn)} \right).
\end{equation}
Moreover, the estimate \eqref{eq:optweak} holds true for any $u \in L^s (0,T; W^{1,s} (\Omega;\RN))$ with an $ s>1$, fulfilling~\eqref{eq:WL} and satisfying $u=g$ on $(0,T)\times \partial \Omega$. The multiplicative constant $C$ depends only on the dimensions $n, N$, $q$, the $\mathcal{C}^1$-modulus of $\Omega$ and the quantities in Assumption~\ref{ass:A}.

In addition, if Assumption~\ref{ass:B} is fulfilled, then for any $s>1$ $u$ is unique in the class $L^s(0,T; W^{1,s}(\Omega;\RN))$. \end{theorem}

\begin{rem}\ \\
The uniqueness statement yields that $u$ is independent of a choice of representative for $\tilde g$.\\ 
Furthermore, for classical datum $\tilde g$, e.g. $\tilde{g}\in L^q(0,T; W^{1-\frac{1}{q},q}(\partial \Omega;\mathbb{R}^N))$ and $\tilde{g}(0)\in L^q(\Omega;\mathbb{R}^N)$ representation $g, F$ along Assumption~\ref{ass:C} follows from the heat flow. Indeed, $g$ can be chosen to be the unique solution to $\partial_t g = \divergence \nabla g$ with initial-boundary datum $\tilde g$. In this case~\eqref{keyest} reduces to
\begin{equation}\label{keyestEx}
\norm{\nabla u}_{L^q(Q_T;\RNn)} \leq C \left(\norm{f}_{L^{q}({Q_T};\RNn)} + \norm{\tilde g (0)}_{L^{q}({Q_T};\RN)}+\norm{\tilde g}_{L^q(0,T; W^{1-\frac{1}{q},q}(\partial \Omega;\RN))}\right).
\end{equation}
\end{rem}



As announced, our Theorem~\ref{thweak} covers the entire range $q \in (1, \infty)$.

To illustrate the generality of right-hand sides admissible there let us denote by $\mathcal{M}(Q_T;\setR^N)$ the space of $\setR^N$-valued Radon measures and consider the problem
 \begin{align}
\label{eq:sysM}
\begin{aligned}
  \partial_tu-\divergence A(z; \nabla u) &= \mu &&\textrm{ in } Q_T,\\
  u&=0&&\textrm{ on } \partial Q_T.
  \end{aligned}
\end{align}
For this setting we have the following existence and uniqueness result:
\begin{corollary}\label{CMB}
Let $\partial \Omega \in C^1$,  ${A}$ satisfy Assumption~\ref{ass:A} and $\mu \in \mathcal{M}(Q_T;\setR^N)$. Then the problem  \eqref{eq:sysM} admits a weak solution $u\in L^s(0, T; W^{1,s}_0(\Omega;\mathbb{R}^{n\times N}))$ for any $1<s<n/(n-1)$.

Moreover, for all $u \in L^s (0,T; W_0^{1,s} (\Omega))$ with $1<s<n/(n-1)$ solving  \eqref{eq:sysM}, the following estimate holds
\begin{equation}\label{eq:optM}
\norm{u}_{L^s (0,T; W_0^{1,s} (\Omega))}  \le C \left(1+ \|\mu\|_{\mathcal{M}(Q_T;\setR^N))} \right).
\end{equation}
In addition, if Assumption~\ref{ass:B} is fulfilled, then the solution $u$ solving \eqref{eq:sysM} is unique in the $L^s(0,T; W^{1,s}(\Omega))$ class, for any $s>1$.
\end{corollary}
%
%

%

Theorem~\ref{thweak} gives the desired optimal result. It follows automatically from the following more general weighted case. Not only the following  result is more general, but it is in fact the key for proving the existence result, since the weighted $L^2$ theory is crucial in our approach. Below, $\omega$ denotes a weight in a  class $\Acal_q$, whereas $\omega' = \omega^{-(q'-1)}$. For more details on Muckenhoupt weights and related spaces, see Section~\ref{S2}.
\begin{theorem}\label{th}
Let $q\in (1,\infty)$ be arbitrary, $\omega \in \Acal_q$ be a Muckenhoupt weight  and $\Omega \in \mathcal{C}^1$. Let ${A}$ satisfy Assumption~\ref{ass:A}, $\tilde{g}$ satisfies Assumption~\ref{ass:C} with $g\in L^1(0,T;W^{1,1}(\Omega;\mathbb{R}^N))$, $\nabla g \in L^q_{\omega}(Q_T;\mathbb{R}^{n\times N})$ and $F\in L^q_{\omega}(Q_T;\mathbb{R}^{n\times N})$ and assume that $f\in L^q_{\omega} (Q_T;\mathbb{R}^{n\times N})$. Then  \eqref{eq:sysA} admits a weak solution $u$ fulfilling
\begin{equation}\label{eq:opt}
\norm{u-g}_{L^1(0,T;W^{1,1}_0(\Omega;\mathbb{R}^{N}))} \; + \norm{\nabla u}_{L^q_{\omega}(Q_T;\mathbb{R}^{n\times N})}\le C \left(1+ \norm{|f| + |\nabla g|+|F|}_{L_\omega^{q} (Q_T)}  \right).
\end{equation}
In addition, the estimate \eqref{eq:opt} holds true for {any} $u \in L^s (0,T; W^{1,s} (\Omega;\RN))$ with an $s>1$, solving  \eqref{eq:sysA} and fulfilling $u=g$ on $(0,T)\times \Omega$. The multiplicative constant $C$ depends only on dimensions $n, N$, $q$, the $\mathcal{C}^1$-modulus of $\Omega$, the constant $A_q(\omega)$ and the quantities in Assumption~\ref{ass:A}.
In addition, if Assumption~\ref{ass:B} is fulfilled, then for any $s>1$ the solution $u$ is unique in the $L^s(0,T; W^{1,s}(\Omega;\RN))$ class.
\end{theorem}

\section{Definitions and auxiliary results}
\label{S2}

\subsection{Muckenhoupt weights and the maximal function}
This section directly rewrites the respective section of \cite{BulBurSch16} to the parabolic setting.
We start this part by recalling the definition of  the 'parabolic` Hardy-Littlewood maximal function. For any $f\in L^1(\setR^{n+1})$ we define
\[
Mf(z):=\sup_{R>0}\dashint_{Q_R(z)}\abs{f(y)}\dy \quad \textrm{with} \quad \dashint_{Q_R(z)}\abs{f(y)}\dy:=\frac{1}{|Q_R(z)|}\int_{Q_R(z)}\abs{f(y)}\dy,
\]
where $B_R(x)$ denotes a ball with radius $R$ centred at $x\in \setR^n$ and $Q_R(z)$ is the respective (parabolic) cylinder, i.e.,  $Q_R(z) := B_R(x) \times (t - R^2, t)$, where $z:=(x,t)\in \setR^{n+1}$. Since such cylinders form a regular family in the sense of Stein, all the standard theory for the Hardy-Littlewood maximal function is valid with respect to them. In this paper we consider our PDE on a bounded domain, hence the involved functions, when maximal function is used, need to be appropriately extended. In most cases extension by $0$ suffices and then we do not distinguish in notation a function $f$ and its trivial extension.

Next, we call $\omega:\setR^{n+1} \to \setR$ a weight, if it is a measurable function that is almost everywhere finite and positive. For such a weight and an arbitrary measurable set $C\subset \setR^{n+1}$ we denote the space $L^p_{\omega}(C)$ with $p\in [1,\infty)$ as
$$
L^p_{\omega}({C}):=\biggset{u:{C} \to \setR^N; \; \norm{f}_{L^p_\omega}
:= \bigg(\int_{{C}} |u(z)|^p\omega(z)\dz\bigg)^{\frac 1p} <\infty}.
$$
In the case $\omega\equiv 1$, the above $L^p_{\omega}$ reduces to the standard Lebesgue space $L^p$. Throughout the paper, we also use the standard notation for Bochner, Sobolev and Sobolev--Bochner spaces. Next, let us introduce classes of Muckenhoupt weights. Note here that our weights are defined on the whole space~$\setR^{n+1}$. For an arbitrary  $p\in [1,\infty)$, we say that a weight $\omega$
belongs to the Muckenhoupt class $\Acal_p$ if and only if
there exists a positive constant $\alpha$ such that for every parabolic cylinder $Q
\subset \setR^{n+1}$ the following holds
\begin{alignat}{2}
  \label{defAp2}
  \left(\dashint_Q
    \omega\dz\right)\left(\dashint_Q\omega^{-(p'-1)}\dz\right)^\frac{1}{p'-1}
  &\le \alpha & \qquad\qquad&\text{if $p \in (1,\infty)$},
  \\
  \label{defA1}
  M\omega(z)&\le \alpha\, \omega(z) &&\text{if $p=1$}.
\end{alignat}
In what follows, we denote by
$A_p(\omega)$ the smallest constant $\alpha$ for which the
inequality~\eqref{defAp2}, resp.~\eqref{defA1}, holds. Due to the celebrated result of Muckenhoupt, see \cite{Mu72}, we know
that $\omega \in \mathcal{A}_p$ is for $1<p<\infty$ equivalent to the
existence of a constant $A'$, such that  for all $f\in L^p(\setR^{n+1})$
\begin{equation}\label{defAp}
\int_{\setR^d} \abs{Mf}^p\omega\dz\leq A'\,\int_{\setR^d} \abs{f}^p\omega\dz.
\end{equation}
Further, if $p \in [1,\infty)$ and $\omega \in \mathcal{A}_p$, then we have an embedding  $L^p_\omega ({C})
\embedding L^1_{\loc}({C})$, since for all cylinders~$Q \subset \setR^{n+1}$ there holds
\begin{align*}
  \dashint_Q \abs{f}\dz &\leq \bigg(\dashint_Q \abs{f}^p
  \omega\dz\bigg)^{\frac 1p} \bigg(\dashint_Q
  \omega^{-(p'-1)}\dz\bigg)^{\frac 1{p'}} \leq \big(
{A}_p(\omega)\big)^{\frac 1p} \bigg( \frac{1}{\omega(Q)}
  \int_Q \abs{f}^p \omega\dz\bigg)^{\frac 1p},
\end{align*}
where $\omega(Q)$ denotes $\int_Q \omega \dz$.
In particular, the distributional derivatives of all $f \in
L^p_\omega$ are well defined. Next, let us summarise some properties of Muckenhoupt weights in the following lemma.
\begin{lemma}[Lemma 1.2.12 in \cite{Tur00}]
  \label{cor:leftopen}
  Let $\omega\in \Acal_p$ for some $p\in [1,\infty)$. Then $\omega\in
  \Acal_q$ for all $q\ge p$. Moreover, there
  exists $s=s(p,A_p(\omega))>1$ such that $\omega \in
  L^s_{\loc}(\setR^d)$ and we have the reverse H\"{o}lder inequality,
  i.e.,
  \begin{equation}
    \left(\dashint_Q \omega^s \dz\right)^{\frac{1}{s}}\le
    C(d, A_p(\omega))\dashint_Q \omega \dz. \label{reversom}
  \end{equation}
  Further, if $p \in (1,\infty)$, then there exists
 $\sigma=\sigma(p,A_p(\omega)) \in [1,p)$ such that
  $\omega\in \mathcal{A}_\sigma$. In addition,   $\omega \in \mathcal{A}_p$ is equivalent to $\omega^{-(p'-1)} \in
  \mathcal{A}_{p'}$.
\end{lemma}
In this paper, we also use the following improved embedding $L^p_\omega ({C})
\embedding L^q_{\loc}({C})$, valid for all $\omega \in \mathcal{A}_p$ with $p \in
(1,\infty)$ and certain $q\in [1,p)$ depending only on $A_p(\omega)$. Such an embedding can be deduced by a direct application of Lemma~\ref{cor:leftopen}. Indeed, since $\omega \in \mathcal{A}_p$, we have $\omega^{-(p'-1)} \in
\mathcal{A}_{p'}$. Thus, via Lemma~\ref{cor:leftopen}, there exists $s=s(A_p(\omega))>1$ such that
\begin{align*}
  \left(\dashint_Q \omega^{-s(p'-1)} \dz\right)^{\frac{1}{s}}\le
  C(A_p(\omega))\dashint_Q \omega^{-(p'-1)} \dz.
\end{align*}
Consequently, for $q:= \frac{sp}{p+s-1} \in (1,p)$ we can use the H\"{o}lder inequality to deduce that
\begin{align}
  \label{eq:lqprop}
  \begin{aligned}
    \bigg( \dashint_Q \abs{f}^q\dz\bigg)^{\frac 1q} &\leq
    \bigg(\dashint_Q \abs{f}^p \omega\dz\bigg)^{\frac 1p}
    \bigg(\dashint_Q \omega^{-s(p'-1)}\dz\bigg)^{\frac 1{s p'}}\\
&\leq C(A_p(\omega)) \bigg( \frac{1}{\omega(Q)} \int_Q
    \abs{f}^p \omega\dz\bigg)^{\frac 1p},
  \end{aligned}
\end{align}
which implies the desired embedding.

The next result makes another link between maximal functions and  $\Acal_p$-weights.
\begin{lemma}[See pages 229--230 in \cite{86:_real} and page 5 in \cite{Tur00}]\label{cor:dual}
Let $f \in L^1_{\loc}(\setR^{n+1})$ be such that $Mf<\infty$ almost everywhere in $\setR^{n+1}$. Then for all $\alpha \in (0,1)$ we have $(Mf)^{\alpha}   \in \Acal_1$. Furthermore, for all $p\in (1,\infty)$ and all $\alpha\in (0,1)$ there holds
$(Mf)^{-\alpha(p-1)}\in \Acal_p$.
\end{lemma}

We would like also  to point out that the maximum $\omega_1 \vee \omega_2$
and minimum $\omega_1 \wedge \omega_2$ of two $\Acal_p$-weights is again an
$\Acal_p$-weight. For $p=2$ we have simply $A_2(\omega_1 \wedge
\omega_2) \leq A(\omega_1) +A_2(\omega_2)$, due to the following straightforward computation
\begin{align}
  \label{eq:A2min}
  \begin{aligned}
    \dashint_Q (\omega_1 \wedge \omega_2) \dz \dashint_Q \frac1{\omega_1
      \wedge \omega_2}\dz &\leq \left[\bigg(\dashint_Q \omega_1\dz\bigg)
    \wedge \bigg( \dashint_Q\omega_2\dz \bigg)\right] \dashint_Q \left(
    \frac{1}{\omega_1} + \frac{1}{\omega_2} \right)\dz
    \\
    &\leq A_2(\omega_1) + A_2(\omega_2).
  \end{aligned}
\end{align}

\subsection{Convergence tools}
In order to identify the limit of approximate problems, possessing only minimal regularity information, we will use among others the following two tools.
\begin{lemma}[Chacon's Biting Lemma, see \cite{BallMurat:89}]\label{thm:blem}
   Let $\Omega$ be a bounded domain in $\setR^{n+1}$ and let
   $\{g^k\}_{k=1}^{\infty}$ be a bounded sequence in $L^1(\Omega)$. Then
   there exists a non-decreasing  sequence of measurable subsets
   $E_j\subset\Omega$ with $|\Omega \setminus E_j|\to 0$ as $j\to \infty$ such that $\{g^k\}_{k\in\mathbb{N}}$ is
   pre-compact in the weak topology of $L^1(E_j)$, for each $j\in\mathbb{N}$.
 \end{lemma}
Note here that in our setting, via Dunford-Pettis theorem, the above pre-compactness of $g^k$ is equivalent  to the following equi-integrability condition:  for every $j\in \mathbb{N}$ and every $\varepsilon>0$ there exists a $\delta_\epsilon>0$ such that for all $A\subset E_j$ with $\abs{A}\leq \delta_\epsilon$ and all $k\in \mathbb{N}$ it holds
\begin{equation}\label{smallunif}
\sup_k\int_A\abs{g^k}\dz\leq \epsilon.
\end{equation}
The next tool is a generalisation of Minty method to the weighted setting.
\begin{lemma}[See pages 4263--4264 in \cite{BulBurSch16}]\label{lemma:Minty} Let $A$ satisfy Assumption~\ref{ass:A} and $\overline{A},C \in L^2_{\omega_0}(Q_T)$ with some $\omega_0 \in \Acal_2$. If
\[
0\leq \int_{Q_T} (\overline{A}-A(z, B))\cdot(C-B) \, \omega_0
\]
for all $B\in L^\infty (Q_T)$, then $\overline{A}(z) = A(z,C(z))$ almost everywhere in $Q_T$.
\end{lemma}

\section{Weighted parabolic Lipschitz truncation}\label{S:Lip}

\noindent
This section is devoted to one of our key tools, namely the so-called parabolic Lipschitz approximation. It is essential to define and identify the nonlinear limit in monotone operator theory when the solution itself is not an admissible test function, which is precisely the main difficulty we deal with in this paper. The basis of the method can be traced back to \cite{KinLew02,DieRuzWol10,Buletall12,BreDieSch13,DieSchStrVer16} see also~\cite[Section~3.2]{BogDuzMin13}. We will follow the approach of~\cite[Theorem 1.1]{DieSchStrVer16}, but there are significant novelties, including a nontrivial extension into the setting of weighted spaces and a more delicate control of the evolutionary term (compare \ref{LS2} below). Our Lipschitz approximation result reads

\begin{theorem}
\label{thm:lipseq}
Let $\Omega \subset \mathbb{R}^n$ be Lipschitz, $T>0$ is given, $Q:=(0,T)\times \Omega$ and   $q\in (1,2)$ arbitrary. Let the sequences $G^k\in L^{q}(Q;\mathbb{R}^{n\times N})$ and $w^k\in L^q(0,T;W^{1,q}_0(\Omega;\mathbb{R}^{N}))$ satisfy
\begin{equation}\label{apkk}
\norm{G^k}_{L^{2}_{\omega}(Q)}+\norm{\nabla w^k}_{L^2_\omega(Q)}+\norm{G^k}_{L^{q}(Q)}+\norm{\nabla w^k}_{L^q (Q)}\leq C
\end{equation}
with an $\Acal_{2}$ weight $\omega$ as well as
\begin{align}
\label{eq:w}
\begin{aligned}
\int_Q w^k\cdot \partial_t \varphi - G^k \cdot \nabla \varphi \dz=0
\end{aligned}
\end{align}
 for all $\varphi \in \mathcal{C}^{0,1}_0((-\infty,T)\times \Omega)$.
Then for arbitrary $\Lambda>1$ there exists a sequence $\{w^k_{\Lambda}\}_{k=1}^{\infty}\subset L^q(0,T; W^{1,q}_0(\Omega;\mathbb{R}^N))$ such that:
\begin{enumerate}[label={\rm (LS\arabic{*})}, leftmargin=*]
\item \label{LS1} The sequence $w^k_{\Lambda}$ satisfies
$$
\|\nabla w^k_{\Lambda}\|_{L^{\infty}(Q)} + \|w^k_{\Lambda}\|_{\mathcal{C}^{\frac12}(\overline{Q})} + \|\partial_t w^k_{\Lambda} \cdot (w^k-w^k_{\Lambda})\|_{L^q(Q)}\le C\Lambda^{4^{\Lambda}}.
$$
  \item
  \label{LS2} We have the following $\Lambda$-independent estimates
   \begin{align*}
   \int_Q \abs{\nabla w^k_{\Lambda}}^q + \abs{\nabla w^k_{\Lambda}}^2\omega + \sqrt{\Lambda} \abs{\partial_t w^k_{\Lambda} \cdot (w^k-w^k_{\Lambda})}\omega\dz &\le C,\\
   \int_Q \abs{w^k_{\Lambda}}^p\dz &\le C\int_{Q}\abs{w^k}^p\dz,
   \end{align*}
      for any $p \in [1, \infty)$.
    \item\label{LS3} For all $\eta\in C^{0,1}_0(Q) $ there holds
    \begin{align*}
     \int_0^T\int_{\Omega} G^k\cdot \nabla (w^k_{\Lambda} \eta)\dz =-\frac12 \int_Q (|w^k_{\Lambda} |^2
      -2w\cdot w^k_{\Lambda}) \partial_t \eta \dz-\int_{Q} (\partial_t w^k_{\Lambda})
      (w^k_{\Lambda}-w)\eta \dz.
    \end{align*}
		\item \label{LS4} If we define the set $\mathcal{O}^k_{\Lambda}:=\{z\in Q; \; w^k(z)\neq w^k_{\Lambda}(z)\}$, then there holds
$$
\abs{\mathcal{O}^k_{\Lambda}}\le \frac{C}{\Lambda}.
$$
    \end{enumerate}
\end{theorem}
The rest of this section is devoted to the proof. First let us develop estimates that hold for a single couple of functions $(w,G)$, satisfying \eqref{eq:w}.
\subsection{Extension}
\label{ssec:extension}
It is convenient for our purpose to use functions which are defined on
the whole space~$\setR \times \setR^n$. Thus, for given $w\in L^q(0,T; W^{1,q}_0(\Omega; \mathbb{R}^N))$ and $G\in L^q(Q;\mathbb{R}^{n\times N})$ fulfilling for all $\varphi \in \mathcal{C}_0^{0,1}((-\infty,T)\times \Omega; \mathbb{R}^N)$
\begin{equation}
\label{eq:ext}
\int_0^T\int_{\Omega} \partial_t \varphi (z)\cdot w(z) - G(z)\cdot \nabla \varphi(z) \dz =0,
\end{equation}
let us define their extensions (keeping the same symbol)  onto $\mathbb{R} \times \mathbb{R}^n$ as follows (recall that $z=(x,t)$)
\begin{equation}\label{extension}
w(t,x):=\left\{
\begin{aligned}
&w(t,x),\\
&w(2T-t,x),\\
&0,
\end{aligned}
\right.
\quad G(t,x):=\left\{
\begin{aligned}
&G(t,x),\\
&-G(2T-t,x),\\
&0,
\end{aligned}\right.
\quad \textrm{ for }
\left\{
\begin{aligned}
(t,x)&\in (0,T)\times\Omega, \\
(t,x)&\in \times \Omega, \\
&\textrm{elsewhwere.}
\end{aligned}\right.
\end{equation}
Then, it directly follows from \eqref{eq:ext} that such an extension fulfills
\begin{equation}
\label{eq:wwhole}
\int_{\mathbb{R}^{n+1}} \partial_t \varphi \cdot w - G\cdot \nabla \varphi \dz =0
\end{equation}
for all $\varphi\in \mathcal{C}^1(\mathbb{R}^{n+1})$ that vanish outside $\Omega$. In addition, we have the estimate
\begin{equation}
\int_{\mathbb{R}^{n+1}}|\nabla w|^q + |G|^q \dz \le C\int_0^T \int_{\Omega} |\nabla w|^q + |G|^q \dz. \label{eq:ext3}
\end{equation}
Moreover, for $\omega \in \Acal_{2}$, let us introduce a new weight
\begin{align}\label{extensionw}
\tilde{\omega}(t,x):= \left\{ \begin{aligned} &\omega(t,x) &&\textrm{if }t\le T,\\
&\omega(2T-t,x) &&\textrm{if } t>T.
\end{aligned}
\right.
\end{align}
With this definition, it is easy to conclude that $\tilde{\omega} \in \Acal_{2}$ and $A_2(\tilde{\omega}) \le C A_2(\omega)$. In addition, we have that
\begin{equation}
\int_{\mathbb{R}^{n+1}}(|\nabla w|^2  + |G|^2)\tilde{\omega} \dz \le C\int_0^T \int_{\Omega} (|\nabla w|^2+ |G|^2) \omega \dz, \label{eq:ext4}
\end{equation}
provided that the right hand side is finite.

%
%
%


\subsection{Whitney covering}
\label{ssec:whitney}
In this section, we use the notation for parabolic cubes $Q$ and corresponding parabolic maximal function $\mathcal{M}$, see Section~\ref{S2}.
First, for given  $\lambda>0$, we define the \emph{bad} set~$\Oal$ as
\begin{align}
  \label{eq:Oal2}
  \Oal := \set{\mathcal{M}(\nabla w)>\lambda} \cup \set{  \mathcal{M}(G) >
    \lambda}.
\end{align}
Note that $\Oal$ is open. According to \cite[Lemma~3.1]{DieRuzWol10} there exists a countable
parabolic Whitney covering $\set{Q_j}_{j\in \mathbb{N}}= \set{I_j \times B_j}_{j\in \mathbb{N}}$ of~$\Oal$, where $B_j\subset \mathbb{R}^n$ are balls of radii $r_j := r_{B_j}$ and correspondingly $|I_j|=r_j^2$ and  the
following holds:
\begin{enumerate}[label={\rm (W\arabic{*})}, leftmargin=*]
\item\label{itm:whit1} $\bigcup_j\frac {1} {2} Q_j   \,=\, \Oal$,
\item\label{itm:whit2} for all $j\in \setN$ we have $8
  Q_j   \subset \Oal$ and $16 Q_j   \cap (\setR^{n+1}\setminus
  \Oal)\neq \emptyset$,
\item \label{itm:whit3} if $ Q_j   \cap Q_k   \neq \emptyset
  $ then $ \frac 12 r_k\le r_j\leq 2\, r_k$,
\item \label{itm:whit3b}$\frac 14 Q_j   \cap \frac 14Q_k   =
  \emptyset$  for all $j \neq k$,
\item \label{itm:whit4}  each $z=(x,t)\in  \Oal $ belongs to at most
  $120^{n+2}$ of the sets $4Q_j$
\end{enumerate}
\noindent Moreover,  there exists a partition of
unity $\set{\rho_j}_{j\in \mathbb{N}} \subset C^\infty_0(\setR^{n+1})$ such that
\begin{enumerate}[label={\rm (P\arabic{*})}, leftmargin=*]
\item \label{itm:P1} $\chi_{\frac{1}{2}Q_j  }\leq \rho_j\leq
  \chi_{\frac 34 Q_j  }$,
\item \label{itm:P3} $\norm{\rho_j}_\infty + r_j \norm{\nabla
    \rho_j}_\infty + r_j^2 \norm{\nabla^2 \rho_j}_\infty +
  r_j^2 \norm{\partial_t \rho_j}_\infty \leq C$
\end{enumerate}
and if for each $k \in \setN$ we define the set $A_k:= \set{ j \,:\, \frac 34
  Q_k   \cap \frac 34 Q_j   \neq \emptyset}$ then
\begin{enumerate}[label={\rm (P\arabic{*})}, leftmargin=*,start=3]
\item \label{itm:P4} $\sum_{j \in A_k} \rho_j = 1$ on $\frac 34 Q_k  $.
\end{enumerate}
Furthermore, we have  the following additional properties
\begin{enumerate}[label={\rm (W\arabic{*})}, leftmargin=*, start=6]
\item \label{itm:whit_fat34} If $j \in A_k$, then $\abs{\frac 34 Q_j
    \cap \frac 34 Q_k  } \geq  32^{-n-2}\max \set{\abs{Q_j  },
    \abs{Q_k  }}$.
\item \label{itm:whit_radii} If $j \in A_k$, then $\frac 12 r_k \leq
  r_j < 2 r_k$.
\item \label{itm:whit_sum} $\# A_k \leq 120^{n+2}$.
\end{enumerate}

\subsection{Approximation}
Here, for given $w$ and $G$ fulfilling \eqref{eq:wwhole} and the corresponding Whitney covering introduced in the previous section, we define the approximation $w_{\lambda}$. To this end let us first introduce the notation for weighted mean values of $w$. Thus, for $\psi\in L^1(\setR^n)$, we set
\[
\mean{w}_{\psi}=\frac{1}{\norm{\psi}_{L^1(\setR^{n+1})}}\int w \psi \dz
\]
and define the mean values $w_j$ corresponding to cubes $Q_j$ by
\begin{align}
  w_j  &:=
  \begin{cases}
    \mean{w}_{\rho_j} &\qquad \text{if $\frac 34 Q_j
      \subset (0,2T) \times \Omega$},
    \\
    0 &\qquad \text{else.}
  \end{cases}\label{def:meanvalue}
\end{align}
Finally, we define our approximation $\wal$ via the formula
\begin{align}
  \label{eq:def_ulambda}
  \wal(x,t) &:= w(x,t) - \sum_{j\in \mathbb{N}} \rho_j(t,x) (w(t,x) - w_j  ).
\end{align}
Due to the property \ref{itm:whit_sum} of the Whitney covering,  the sum is  well
defined for almost all $z\in \mathbb{R}^{n+1}$. In addition, due to the definition of mean values $w_j$ and the fact that $w\equiv 0$ outside of $(0,2T)\times \Omega$, we see that  $\sum_j \rho_j (w - w_j  )$ is zero outside
of~$(0,2T) \times \Omega$ as well. Consequently,  also~$\wal$ is zero outside of
$(0,2T) \times \Omega$. In fact, we even have
\begin{align}
  \label{eq:supp_bad_parts}
  \support(\rho_j (w - w_j  )) \subset \tfrac 34 Q_j
  \cap ((0,2T) \times \Omega).
\end{align}
Indeed, $\support \rho_j \subset \frac 34 Q_j  $, so the case
$\frac 34 Q_j   \subset (0,2T) \times \Omega$ is obvious. If $\frac
34 Q_j   \not\subset (0,2T) \times \Omega$, then $w_j  =0$
and the claim follows by $\support \rho_j \subset \frac 34
Q_j$ and $\support w \subset (0,2T) \times \Omega$. So, \eqref{eq:supp_bad_parts} follows.

We are ready to introduce the key lemmata. The first one is, up to minor modifications, proved in~\cite{DieSchStrVer16}. The second one is of the same character, but we provide a detailed proof.
\begin{lemma}[Lemma 3.1~\cite{DieSchStrVer16}]
Let $w$ and $G$ satisfy \eqref{eq:ext}, their extension be defined through \eqref{extension}, the Whitney covering be defined with the help of the set $\Oal$ as in \eqref{eq:Oal2} and mean values  be defined via \eqref{def:meanvalue}. Then
  \label{lem:diff_uj_uk}
  \begin{align*}
    \sum_{j \in A_k} \frac{\abs{w_j-w_k}}{r_j} &\leq c\, \sum_{j \in
      A_k} \dashint_{\frac 34 Q_j} \frac{\abs{w-w_j}}{r_j} \dz \leq
    c\, \lambda.
  \end{align*}
\end{lemma}

\begin{lemma}[Pointwise Poincar\'{e} inequality]
  \label{lem:Poinc}Let all assumptions of Lemma~\ref{lem:diff_uj_uk} be satisfied. Then there exists a constant $C$ depending only on $T$ and $\Omega$ such that for any $Q_j$ from the Whitney covering, any $k\in A_j$ and almost all  $z\in \frac34 Q_j$ there holds
  \begin{align}\label{poinc:point}
     \frac{\abs{w(z)-w_k}}{r_j} &\leq C(\mathcal{M}(\nabla w)(z) + \mathcal{M}(G)(z)).
  \end{align}
\end{lemma}
\begin{proof}
First of all, we notice that it is just enough to prove that
  \begin{align}\label{poinc:pointa}
     \frac{\abs{w(z)-w_j}}{r_j} &\leq C(\mathcal{M}(\nabla w)(z) + \mathcal{M}(G)(z)).
  \end{align}
Indeed, using  Lemma~\ref{lem:diff_uj_uk}, we have for all $k\in A_j$ that
$$
  \frac{\abs{w(z)-w_k}}{r_j} \leq \frac{\abs{w(z)-w_j}}{r_j}+\frac{\abs{w_j-w_k}}{r_j}\le   C\lambda + \frac{\abs{w(z)-w_j}}{r_j}.
$$
Since $z\in \Oal$, we can use \eqref{eq:Oal2} to conclude
$$
  \frac{\abs{w(z)-w_k}}{r_j} \leq C(\mathcal{M}(\nabla w)(z) + \mathcal{M}(G)(z)) + \frac{\abs{w(z)-w_j}}{r_j}.
$$
Thus, to show \eqref{poinc:point} for any $k\in A_j$, it suffices to show \eqref{poinc:pointa}.

Next we proceed formally, but all steps can be justified by using the proper mollification. For simplicity, let us assume that $Q_j:=B_{r_j}(0)\times (-r_j^2,0)$. In case that $\frac34 Q_j \subset (0,2T)\times \Omega$, we have for any $z_0 = (x_0,t_0) \in \frac34 Q_j$
$$
\begin{aligned}
&w_j-w(z_0) 
= \int_0^1 \frac{d}{d\varepsilon} \left(\frac{\int_{\mathbb{R}^{n+1}}\rho_j(x, t) w(x_0 +\varepsilon (x-x_0), t_0+\varepsilon^2(t -t_0))\dx \dt }{\int_{\mathbb{R}^{n+1}}\rho_j(x, t)\dx \dt} \right) \deps\\
&= \frac{1}{\|\rho_j\|_1} \int_0^1  \int_{\mathbb{R}^{n+1}}\rho_j(x, t) \nabla w(x_0 +\varepsilon (x-x_0), t_0+\varepsilon^2(t -t_0))\cdot (x-x_0)\dx \dt \deps\\
&\quad + \frac{2}{\|\rho_j\|_1} \int_0^1  \int_{\mathbb{R}^{n+1}}\rho_j(x, t) \partial_t w(x_0 +\varepsilon (x-x_0), t_0+\varepsilon^2(t -t_0))\varepsilon(t-t_0)\dx \dt \deps\\
&= \frac{1}{\|\rho_j\|_1} \int_0^1  \varepsilon^{-n-3}\int_{\mathbb{R}^{n+1}}\rho_j(\frac{x-x_0}{\varepsilon}, \frac{t-t_0}{\varepsilon^2}) \nabla w(x-\varepsilon x_0, t-\varepsilon^2 t_0)\cdot(x-(1+\varepsilon)x_0)\dx \dt \deps\\
&\quad + \frac{2}{\|\rho_j\|_1} \int_0^1   \varepsilon^{-n-3} \int_{\mathbb{R}^{n+1}}\rho_j(\frac{x +\varepsilon x_0 -x_0}{\varepsilon}, \frac{t + \varepsilon^2t_0 -t_0}{\varepsilon^2}) \partial_t w(x, t)(t  -t_0)\dx \dt \deps\\
&= \frac{1}{\|\rho_j\|_1} \int_0^1  \varepsilon^{-n-3}\int_{\mathbb{R}^{n+1}}\rho_j(\frac{x-x_0}{\varepsilon}, \frac{t-t_0}{\varepsilon^2}) \nabla w(x-\varepsilon x_0, t-\varepsilon^2 t_0)\cdot(x-(1+\varepsilon)x_0)\dx \dt \deps\\
&\quad -\frac{2}{\|\rho_j\|_1} \int_0^1   \varepsilon^{-n-3} \int_{\mathbb{R}^{n+1}}\frac{\partial}{\partial t} \left(\rho_j(\frac{x +\varepsilon x_0 -x_0}{\varepsilon}, \frac{t + \varepsilon^2t_0 -t_0}{\varepsilon^2})(t-t_0)\right)  w(x, t)\dx \dt \deps\\
&=:I_1 + I_2.
\end{aligned}
$$
Finally, we shall estimate terms $I_1$ and $I_2$. Using \ref{itm:P1}, the fact that $Q_j$ is centered in zero and that $x_0 \in B_{\frac34 r_j}(0)$, we obtain
$$
\begin{aligned}
|I_1|&\le  \frac{2^{n+2}}{r_j^{n+2}} \int_0^1  \varepsilon^{-n-3}\int_{Q_{\frac34 \varepsilon r_j}(x_0,t_0)}|\nabla w(x-\varepsilon x_0, t-\varepsilon^2 t_0)||x-(1+\varepsilon)x_0|\dx \dt \deps\\
&\le  \frac{3\cdot 2^{n+1}}{r_j^{n+1}} \int_0^1  \varepsilon^{-n-2}\int_{Q_{\frac34 \varepsilon r_j}(x_0,t_0)}|\nabla w(x-\varepsilon x_0, t-\varepsilon^2 t_0)|\dx \dt \deps\\
&\le  \frac{3\cdot 2^{n+1}}{r_j^{n+1}} \int_0^1  \varepsilon^{-n-2}\int_{Q_{2\varepsilon r_j}(x_0,t_0)}|\nabla w(x, t)|\dx \dt \deps\\
&\le  3\cdot 4^{n+2}r_j \int_0^1  \mathcal{M}(\nabla w)(z) \deps=3\cdot 4^{n+2}r_j \mathcal{M}(\nabla w)(z).
\end{aligned}
$$
To estimate $I_2$, we intend to use \eqref{eq:wwhole}. In order to do so we need to show that for $x\notin \Omega$ the function $\rho_j(\frac{x +\varepsilon x_0 -x_0}{\varepsilon}, \frac{t + \varepsilon^2t_0 -t_0}{\varepsilon^2})$ vanishes. Actually, we will show that it vanishes whenever $x\notin B_{\frac34 r_j}(0)$ and then due to the fact that $B_{\frac34 r_j}(0)\subset \Omega$ the claim follows. Indeed, this fact directly follows from the triangle inequality, the fact that $x_0\in B_{\frac34 r_j}$ and the property \ref{itm:P1} of $\rho_j$. Hence, we can use \eqref{eq:wwhole} and with the help of \ref{itm:P1} and \ref{itm:P3} we deduce that
$$
\begin{aligned}
|I_2|&\le  \frac{2}{\|\rho_j\|_1} \int_0^1   \varepsilon^{-n-4} \int_{\mathbb{R}^{n+1}} \left|\nabla \rho_j(\frac{x +\varepsilon x_0 -x_0}{\varepsilon}, \frac{t + \varepsilon^2t_0 -t_0}{\varepsilon^2})\right||G (x, t)| |t-t_0|\dx \dt \deps\\
&\le  \frac{C(n)}{r_j^{n+2}} \int_0^1   \varepsilon^{-n-4} \int_{Q_{\frac34 \varepsilon r_j}(x_0,t_0)}r_j^{-1}|G(x, t)| \varepsilon^2 r_j^2\dx \dt \le  C(n)r_j \mathcal{M}(G)(z).
\end{aligned}
$$
Combining the estimates for $I_1$ and $I_2$, we immediately obtain \eqref{poinc:point}.

In the case $\frac34 Q_j \nsubseteq (0,2T)\times \Omega$, we can in fact use almost the same procedure, since $w$ is extended outside $(0,2T)\times \Omega$ by zero, $\Omega$ is Lipschitz and thus we can use the proper version of the Poincar\'{e} inequality.
\end{proof}

At this point we have collected all the auxiliary results needed for the following theorem, which is a generalization of \cite[Theorem 1.1]{DieSchStrVer16}. The key extension, which will serve as the backbone for Theorem~\ref{thm:lipseq}, involves estimates in weighted spaces.
\begin{theorem}
  \label{thm:lip}
Let $\Omega \subset \mathbb{R}^n$ be Lipschitz, $T>0$ is given, $Q:=(0,T)\times \Omega$ and   $q\in (1,\infty)$ arbitrary. Assume that $G\in L^q(Q; \mathbb{R}^{n \times N})$ and $w\in L^q(0,T; W^{1,q}_0(\Omega;\mathbb{R}^N))$ satisfy \eqref{eq:ext}. Further, let $\omega$ be an $\Acal_{2}$ weight and define the extensions for $w$, $G$ and $\omega$ by \eqref{extension}. For arbitrary $\lambda>0$ let us define
\begin{align*}
\Oal:=\bigset{M(\nabla w)>\lambda}\cup\bigset{M(G)>\lambda}.
\end{align*}
Then there exists an approximation $\wal\in
  L^{\infty}(\mathbb{R};W^{1,\infty}_0(\Omega;\mathbb{R}^N))$  with the following properties:
  \begin{enumerate}[label={\rm (L\arabic{*})}, leftmargin=*]
  \item
  \label{itm:oal} $\wal=w$ on $({\Oal})^c\cup ((0,2T)\times \Omega)^c$, in particular $\wal=0$ on $\mathbb{R}^{n+1}\setminus (0,2T)\times \Omega$.
  \item
  \label{itm:stab}
  There is a constant $c$ depending only on $\Omega$ and $q$ such that
  \[
  \int_{(0,2T)\times\Omega}\abs{\nabla(w-\wal)}^q \dz= \int_{\Oal}\abs{\nabla (w-\wal)}^q \dz\leq c\int_{\Oal}\abs{\nabla w}^q+\abs{G}^q \dz.
  \]
     \item
  \label{itm:max} $\mathcal{M}(\nabla \wal)\leq c\,\lambda$, i.e., $\wal$ is
    Lipschitz continuous with respect to the spatial variable.
  \item
  \label{itm:lip}
  $\wal$ is Lipschitz continuous with respect to the
    parabolic metric, i.e.,
    \begin{align*}
      \abs{\wal(t,x) - \wal(s,y)} \leq c\, \lambda \max
      \biggset{ \abs{t-s}^\frac 12,\abs{x-y} }
    \end{align*}
    for all $(t,x),(s,y)\in (0,2T)\times \Omega$.
    \item\label{itm:integrationbyparts1} For all $\eta\in C^{0,1}_0(Q) $ there holds
    \begin{align*}
     \int_0^T\int_{\Omega} G\cdot \nabla (\wal \eta)\dz =-\frac12 \int_Q (|\wal |^2
      -2w\cdot \wal) \partial_t \eta \dz-\int_{\Oal} (\partial_t \wal)
      (\wal-w)\eta \dz.
    \end{align*}
		\item \label{itm:liptime1} For every $\lambda$ the quantity $\partial_t \wal \cdot
      (\wal-w)$ belongs to $L^1(\mathbb{R}^{n+1})$ and we have
		\[
\int \abs{\partial_t \wal
      (\wal-w)}^q\dz\leq c\lambda^{q}\int_{Q}\abs{\nabla w}^q+\abs{G}^q\dz.
\]
    \item \label{itm:weight} If $\nabla w, G \in L^2_\omega((0,T)\times \Omega)$ then we have
   \begin{align*}
      \int_{Q}\abs{\nabla w_{\lambda}}^2\omega \dz
      \leq c\int_{Q}(\abs{\nabla w}^2+\abs{G}^2)\omega \dz.
    \end{align*}
			\item\label{itm:liptime2}
			\[
			   \begin{aligned}
			\Bigabs{\int_Q (\partial_t \wal)
      (\wal-w)\omega \dz} &\leq  C\sqrt{\lambda^2\tilde{\omega}(\Oal)} \left(\int_{Q}(\abs{\nabla w}^2      +\abs{G}^2)\omega \dz\right)^{\frac12} \\
			   \int_Q \abs{w_{\lambda}}^p\dz &\le C\int_{Q}\abs{w}^p\dz
		    \end{aligned}
		    \]
		    for any $p \in [1, \infty)$.
    \end{enumerate}
\end{theorem}
\begin{proof}
Properties \ref{itm:oal}--\ref{itm:integrationbyparts1} are exactly stated in~\cite[Theorem 1.1]{DieSchStrVer16}. The remainder needs to be proven. First, let us focus on~\ref{itm:weight}. Using the definition \eqref{eq:def_ulambda}, we have for an arbitrary $z\in Q_k$ that
\begin{align*}
\abs{\nabla (w(z) - \wal(z))}&=\abs{\nabla \sum_{j\in A_k} \rho_j(z) (w(z) - w_j  )}\le \sum_{j\in A_k} \abs{\nabla  \rho_j(z)} \abs{w(z) - w_j}+\rho_j\abs{\nabla  w(z)}\\
\intertext{and using \ref{itm:P1}, \ref{itm:P3}, \ref{itm:whit_radii} and \eqref{poinc:point}, we obtain}
\abs{\nabla (w(z) - \wal(z))} &\le C\sum_{j\in A_k} \frac{\abs{w(z) - w_j}}{r_j}\chi_{\frac34Q_j}+\abs{\nabla  w(z)}\le C(\mathcal{M}(\nabla w)(z)+\mathcal{M}(G)(z)).
\end{align*}
Thus using the continuity of the maximal function, see \eqref{defAp}, the fact that $\omega \in \Acal_{2}$ and the extensions \eqref{extension} and \eqref{extensionw}, we have
$$
\begin{aligned}
\int_Q \abs{\nabla \wal}^2 \omega \dz &\le C\int_{\Oal} \abs{\nabla(w-\wal)}^2 \tilde{\omega} \dz + C\int_Q \abs{\nabla w}^2\omega \dz\\
 &\le C\int_{\mathbb{R}^{n+1}}\left((\mathcal{M}(\nabla w))^2 +(\mathcal{M}(G))^2 \right)\tilde{\omega} \dz \le C\int_{\mathbb{R}^{n+1}}\left(\abs{\nabla w}^2 +\abs{G}^2 \right)\tilde{\omega} \dz\\
 &\le C\int_{Q}\left(\abs{\nabla w}^2 +\abs{G}^2 \right)\omega \dz,
\end{aligned}
$$
which is \ref{itm:weight}.

Secondly, we focus on the terms with time derivatives, i.e., on the proof of~\ref{itm:liptime1} and   \ref{itm:liptime2}. Since $\frac34Q_i\subset \Oal$, we have by the definition of the Lipschitz truncation and by $\sum_{j\in A_i}\rho_j\equiv 1$, see \ref{itm:P4}, that for any $z\in \frac34Q_i$
\begin{align*}
(\partial_t\wal(z))
      (\wal(z)-w(z))&=\partial_t \Big(\sum_{j\in A_i} \rho_j(z) w_j\Big)\sum_{m\in A_i}(w_m-w(z))\rho_m(z)
      \\
      &= \sum_{j\in A_i} \partial_t\rho_j(z) (w_j-w_i)\sum_{m\in A_i}(w_m-w(z))\rho_m (z).
\end{align*}
Hence, using \ref{itm:P3}, \ref{itm:whit_radii}, Lemma~\ref{lem:diff_uj_uk} and Lemma~\ref{lem:Poinc}, we obtain for any $z\in\frac34 Q_i$
\begin{align}
\begin{aligned}
|\partial_t\wal(z)
      (\wal(z)-w(z))|&\le  C\sum_{j\in A_i}\sum_{m\in A_i} \frac{|w_j-w_i|}{r_j}\frac{|w_m-w(z)|}{r_j}\chi_{\frac34 Q_m}\\
      &\le  C\lambda \left(\mathcal{M}(\nabla w)(z) +\mathcal{M}(G)(z)\right).
\end{aligned}\label{time:est}
\end{align}
Consequently, we obtain with the help of \ref{itm:whit1} and the continuity of the maximal function that
\begin{align*}
\int_Q \abs{\partial_t \wal
      (\wal-w)}^q\dz&\le \int_{\Oal}\abs{\partial_t \wal
      (\wal-w)}^q\dz\le \sum_i\int_{\frac34 Q_i}\abs{\partial_t \wal
      (\wal-w)}^q\dz \\
      &\le C\lambda^q \sum_i\int_{\frac34 Q_i}\left(\mathcal{M}(\nabla w) +\mathcal{M}(G)\right)^q \dz\\
      &\le C\lambda^q \int_{\mathbb{R}^{n+1}}\left(\mathcal{M}(\nabla w) +\mathcal{M}(G)\right)^q \dz\le C\lambda^q \int_{\mathbb{R}^{n+1}}\abs{\nabla w}^q      +\abs{G}^q \dz\\
      &\le C\lambda^q \int_{Q}\abs{\nabla w}^q      +\abs{G}^q \dz.
\end{align*}
This proves \ref{itm:liptime1}. It remains to show  \ref{itm:liptime2}. Having already the estimate \eqref{time:est}, we can use the continuity of the maximal function in weighted spaces, the fact that $\omega$ is an $\Acal_{2}$ weight, the H\"{o}lder inequality and the properties of the Whitney covering to deduce
\begin{align*}
\int_Q \abs{\partial_t \wal
      (\wal-w)}\omega\dz&\le \int_{\Oal}\abs{\partial_t \wal
      (\wal-w)} \tilde \omega\dz\le \sum_i\int_{\frac34 Q_i}\abs{\partial_t \wal
      (\wal-w)}\tilde{\omega}\dz \\
      &\le C\lambda \sum_i\int_{\frac34 Q_i}\left(\mathcal{M}(\nabla w) +\mathcal{M}(G)\right)\tilde{\omega} \dz\\
      &\le C\lambda \int_{\Oal}\left(\mathcal{M}(\nabla w) +\mathcal{M}(G)\right)\tilde{\omega} \dz\\
      &\le C\sqrt{\lambda^2\tilde{\omega}(\Oal)}\left(\int_{\Oal}\left(\mathcal{M}(\nabla w) +\mathcal{M}(G)\right)^2\tilde{\omega} \dz \right)^{\frac12} \\
      &\le C\sqrt{\lambda^2\tilde{\omega}(\Oal)} \left(\int_{Q}(\abs{\nabla w}^2      +\abs{G}^2)\omega \dz\right)^{\frac12}.
\end{align*}
%
For the latter statement of  \ref{itm:liptime2} it clearly suffices to show
\[
\int_{\Oal} \abs{w-w_{\lambda}}^p = \int_Q \abs{w-w_{\lambda}}^p \le C \int_Q \abs{w}^p.
\]
Hence we write, using the definition \eqref{eq:def_ulambda} of $w_\lambda$ and properties of partition of unity: $\rho_j \in [0,1]$, the  finite intersection property \ref{itm:whit_sum}
\[
\begin{aligned}
 \int_{\Oal} \abs{w-w_{\lambda}}^p &\le  \sum_{i\in \mathbb{N}} \int_ {\frac12 Q_i} \bigg| \sum_{j \in A_i} \rho_j(z) (w(z) - w_j  ) \bigg|^p dz \le  C \sum_{i\in \mathbb{N}} \int_ {\frac12 Q_i} \sum_{j \in A_i} \rho_j(z) |w(z) - w_j|^p dz \\
 &\le C \int_Q \abs{w}^p\dz + C \sum_{i\in \mathbb{N}} \int_ {\frac12 Q_i} \sum_{j \in A_i} \rho_j(z) \frac{\int \abs{ w }^p \rho_j}{\int \rho_j} dz \\
 &\le C \int_Q \abs{w}^p\dz + C \sum_{i\in \mathbb{N}} \int_ {\frac12 Q_i} \sum_{j \in A_i} \rho_j(z) \frac{\int_{Q_i} \abs{ w }^p \rho_j}{\int \rho_j} dz
 \\
 &\le C \int_Q \abs{w}^p\dz + C \sum_{i\in \mathbb{N}} \left(\int_{Q_i} \abs{ w }^p \right)  \sum_{j \in A_i} \int_ {\frac12 Q_i} \rho_j(z) \frac{1}{\int \rho_j} dz \le  C \int_{2Q} \abs{w}^p\dz,
 \end{aligned}
\]
where in the second line we invoked the definition \eqref{def:meanvalue} of the weighed mean value $w_j$ and the related Jensen inequality and in the third line \ref{itm:whit3}. The right hand side and inequality for extension imply the second part of  \ref{itm:liptime2}. The proof of Theorem \ref{thm:lip} is complete.
\end{proof}

\subsection{Proof of Theorem~\ref{thm:lipseq}}
Recall that up to now we have studied the case of a single couple of functions $(w,G)$ satisfying \eqref{eq:w}. In order to make the step to a sequence $(w^k,G^k)$, let us first consider an arbitrary $\Lambda>0$ and $m_0\in \mathbb{N}$. Due to the continuity of the maximal function and the fact that $\omega$ is an $\Acal_{2}$ weight, we deduce from \eqref{apkk} that (here we extended all quantities outside $(0,T)\times \Omega$ by \eqref{extension} and \eqref{extensionw})
\begin{equation}\label{final_start}
\int_{\mathbb{R}^{n+1}} \abs{\mathcal{M}(G^k)}^q + \abs{\mathcal{M}(\nabla w^k)}^q + \abs{\mathcal{M}(G^k)}^2\tilde{\omega} + \abs{\mathcal{M}(\nabla w^k)}^2\tilde{\omega} \dz \le C.
\end{equation}
Thus, using this estimate, we have that for arbitrary $k$
$$
\begin{aligned}
\min_{m\in \{0,\ldots, m_0\}}& \left(\int_{\{\Lambda^{2^m}< \mathcal{M}(G^k)\le \Lambda^{2^{m+1}} \}} \abs{\mathcal{M}(G^k)}^q  + \abs{\mathcal{M}(G^k)}^2\tilde{\omega} \dz\right.\\
 &\quad \left. +\int_{\{\Lambda^{2^m}<\mathcal{M}(\nabla w^k)\le \Lambda^{2^{m+1}}\}} \abs{\mathcal{M}(\nabla w^k)}^q  + \abs{\mathcal{M}(\nabla w^k)}^2\tilde{\omega} \dz \right)\\
 & \le \frac{1}{m_0}  \sum_{m=0}^{m_0}\left(\int_{\{\Lambda^{2^m}< \mathcal{M}(G^k)\le \Lambda^{2^{m+1}} \}} \abs{\mathcal{M}(G^k)}^q  + \abs{\mathcal{M}(G^k)}^2\tilde{\omega} \dz\right.\\
 &\quad \left. +\int_{\{\Lambda^{2^m}<\mathcal{M}(\nabla w^k)\le \Lambda^{2^{m+1}}\}} \abs{\mathcal{M}(\nabla w^k)}^q  + \abs{\mathcal{M}(\nabla w^k)}^2\tilde{\omega} \dz \right)\\
 & \le \frac{1}{m_0}  \int_{Q} \abs{\mathcal{M}(G^k)}^q  + \abs{\mathcal{M}(G^k)}^2\tilde{\omega} + \abs{\mathcal{M}(\nabla w^k)}^q  + \abs{\mathcal{M}(\nabla w^k)}^2\tilde{\omega} \dz \le \frac{C}{m_0}
\end{aligned}
$$
Thus for every $k$ we can find $m_k\in \{0,\ldots, m_0\}$ such that
\begin{align}\label{findm0}
\begin{aligned}
&\left(\int_{\{\Lambda^{2^{m_k}}< \mathcal{M}(G^k)\le \Lambda^{2^{m_k+1}} \}} \abs{\mathcal{M}(G^k)}^q  + \abs{\mathcal{M}(G^k)}^2\tilde{\omega} \dz\right.\\
 &\quad \left. +\int_{\{\Lambda^{2^{m_k}}<\mathcal{M}(\nabla w^k)\le \Lambda^{2^{m_k+1}}\}} \abs{\mathcal{M}(\nabla w^k)}^q  + \abs{\mathcal{M}(\nabla w^k)}^2\tilde{\omega} \dz \right)\le \frac{C}{m_0}.
\end{aligned}
\end{align}
Hence defining
$$
\lambda_k:=\Lambda^{2^{m_k}},
$$
we can use \eqref{final_start} and \eqref{findm0} to observe
$$
\begin{aligned}
\lambda_k^q|\mathcal{O}_{\lambda_k}|+\lambda_k^2 \tilde{\omega}(\mathcal{O}_{\lambda_k})&\le \int_{\{\mathcal{M}(G^k)>\lambda_k\}} \lambda_k^q + \lambda_k^2 \tilde{\omega}\dz +\int_{\{\mathcal{M}(\nabla w^k)>\lambda_k\}} \lambda_k^q + \lambda_k^2 \tilde{\omega}\dz\\
&\le \int_{\{\lambda_k^2\ge \mathcal{M}(G^k)>\lambda_k\}} \lambda_k^q + \lambda_k^2 \tilde{\omega}\dz +\int_{\{\lambda_k^2\ge \mathcal{M}(\nabla w^k)>\lambda_k\}} \lambda_k^q + \lambda_k^2 \tilde{\omega}\dz\\
&\quad + \int_{\{\mathcal{M}(G^k)>\lambda^2_k\}} \lambda_k^q + \lambda_k^2 \tilde{\omega}\dz +\int_{\{\mathcal{M}(\nabla w^k)>\lambda^2_k\}} \lambda_k^q + \lambda_k^2 \tilde{\omega}\dz\\
&\le \int_{\{\lambda_k^2\ge \mathcal{M}(G^k)>\lambda_k\}} \abs{\mathcal{M}(G^k)}^q + \abs{\mathcal{M}(G^k)}^2 \tilde{\omega}\dz \\
&\quad +\int_{\{\lambda_k^2\ge \mathcal{M}(\nabla w^k)>\lambda_k\}} \abs{\mathcal{M}(\nabla w^k)}^q + \abs{\mathcal{M}(\nabla w^k)}^2 \tilde{\omega}\dz\\
&\quad + \int_{\mathbb{R}^{n+1}} \lambda_k^{-q}\abs{\mathcal{M}(G^k)}^q + \lambda_k^{-2}\abs{\mathcal{M}(G^k)}^2 \tilde{\omega}\dz \\
&\quad +\int_{\mathbb{R}^{n+1}} \lambda_k^{-q}\abs{\mathcal{M}(\nabla w^k)}^q + \lambda_k^{-2}\abs{\mathcal{M}(\nabla w^k)}^2 \tilde{\omega}\dz\\
&\le \frac{C}{m_0} + \frac{C}{\lambda_k}.
\end{aligned}
$$
Thus, setting $m_0:=\Lambda$, we obtain
\begin{align}\label{Burcz}
\lambda_k^q|\mathcal{O}_{\lambda_k}|+\lambda_k^2 \tilde{\omega}(\mathcal{O}_{\lambda_k})&\le \frac{C}{\Lambda}.
\end{align}

Now, for a fixed $k$ we use $\lambda:=\lambda_k$, $G:=G^k$ and $w:=w^k$ in Theorem~\ref{thm:lip} and define $w^k_{\Lambda}:=w^k_{\lambda_k}$. Since $\Lambda\le \lambda_k \le \Lambda^{4^{\Lambda}}$ and we have the uniform bound \eqref{apkk}, we see that \ref{LS1} follows from \ref{itm:lip}, \ref{itm:liptime1} and \ref{itm:max}. Similarly, \ref{LS2} follows from \ref{itm:stab},  \ref{itm:weight} and \ref{itm:liptime2} combined with \eqref{Burcz}. Then \ref{LS3} is nothing else than \ref{itm:integrationbyparts1}. Finally, since $\mathcal{O}^k_{\Lambda} \subset \mathcal{O}_{\lambda_k}$, the claim \ref{LS4} follows from \eqref{Burcz}. The proof is complete.


\section{Linear theory}
This section is devoted to the linear theory, i.e., to the theory for \eqref{eq:sysA} with $A$ being linear with respect to $\nabla u$. More precisely, cf. \eqref{tildeA}, we assume that a given $\tilde{A}:L^{\infty}(Q_T; \RNn \times \RNn)$ is \emph{strongly elliptic}, i.e., there exist positive constants $c_1$ and $c_2$ such that for $\eta \in  \mathbb{R}^{n\times N}$
\begin{equation}\label{str.elip}
c_1|\eta|^2\le {\tilde A (z)} \eta \cdot \eta \le c_2|\eta|^2.
\end{equation}
and we analyse
\begin{align}
\label{eq:sysL}
\begin{aligned}
  \partial_t u-\divergence( \tilde A \nabla u) &= -\divergence f &&\textrm{ in }{Q_T},\\
  u&=\tilde{g}&&\textrm{ on } \partial {Q_T},
  \end{aligned}
\end{align}
with datum $\tilde{g}$ represented by certain $g$ and $F$ along Assumption~\ref{ass:C}. The main result of this section reads
\begin{theorem}
\label{thm:lin}
Let $\partial \Omega \in \mathcal{C}^1$, $q\in (1,\infty)$, $\omega \in \Acal_q$ and $\tilde A \in
  \mathcal{C}(\overline{{Q_T}}, \setR^{n\times N \times n\times N})$  strongly elliptic be given. Consider the data $f\in L^q_\omega ({Q_T};\setR^{n\times N})$ and $\tilde g$ satisfying the representation Assumption ~\ref{ass:C}, which $\nabla g \in L^q_\omega ({Q_T};\setR^{n\times N})$ and $F\in L^q_\omega ({Q_T};\setR^{n\times N})$. Then there exists $u\in L^1(0,T; W^{1,1}(\Omega;\RN))$ fulfilling for all $\varphi \in \mathcal{C}^1_0((-\infty,T)\times \Omega;\RN)$
\begin{equation}\label{wf:lin}
\int_{Q_T} -(u-g)\cdot \partial_t \varphi + \tilde{A} \nabla u \cdot \nabla \varphi \dz = \int_{Q_T} (f+F)\cdot \nabla \varphi \dz.
\end{equation}
Moreover, we have that $(u-g)\in L^1(0,T; W^{1,1}_0(\Omega;\RN))$ and the following estimate holds
\begin{equation}\label{keyest}
\norm{\nabla u}_{L^q_{\omega}(Q_T)} \leq C(c_1, c_2,q, \tilde{A},\Omega, \Acal_q(\omega))\left(\norm{f}_{L_\omega^{q}({Q_T})} + \norm{\nabla g}_{L_\omega^{q}({Q_T})}+\norm{F}_{L_\omega^{q}({Q_T})}\right).
\end{equation}
In addition, let $s>1$, then $u$ is unique in the $L^s(0,T; W^{1,s}(\Omega))$ class (in the sense of subsection \ref{ssec:uq}). 

\end{theorem}

Let us first recall the respective result for homogeneous data and within $L^q$ spaces (without Muckenhoupt weights). It will play an essential role in proving Theorem~\ref{thm:lin}.
\begin{lemma}
  \label{lem:CZ}
  Let $\partial \Omega \in \mathcal{C}^1$  and $\tilde{A} \in
  \mathcal{C}(\overline{{Q_T}}, \setR^{n\times N \times n\times N})$ be strongly elliptic. Then, for any $f\in L^q({Q_T})$, $q \in (1, \infty)$, there exists a weak solution to \eqref{eq:sysL} with $\tilde{g}=0$. In addition, it satisfies
  \begin{equation}\label{aoprioriconst}
       \norm{\nabla u}_{L^{q}({Q_T})} \leq C(\tilde{A},q,\Omega)\norm{f}_{L^{q}({Q_T})}.
  \end{equation}
  Moreover, it is unique in the class \[
  u\in
  L^q (0, T; W^{1,q}_0(\Omega)).
  \]
\end{lemma}
\begin{proof}
In the case of a single equation, the above lemma, even under considerably more general assumptions on domain (locally flat Lipschitz condition) and regularity of $A$ (small BMO), can be found as Theorem 1.5   in  \cite{Byun05jde}, the case for systems follows from~\cite{Byunatall15} using the weight $\omega\equiv 1$. Observe, that the assumption on the weights in ~\cite{Byunatall15} namely $\omega\in \Acal_{q/2}$, is essentially stronger then what is assumed above in Theorem~\ref{thm:lin}.
The uniqueness for $q>2$ is automatic and the case $q<2$ follows from duality.
\end{proof}

The rest of this section is devoted to the proof of Theorem~\ref{thm:lin}. In the first part, we focus on the proof for the homogeneous case, i.e. the case when $\tilde{g}=0$. Next, we apply the result for homogeneous case to the inhomogeneous setting thus obtaining Theorem~\ref{thm:lin} in its full generality.

\subsection{Homogeneous data}
First, we observe that~\eqref{eq:lqprop} and boundedness of $ {Q_T}$ implies that $\exists \;  {q_0} \in(1, p)$ such that $f \in L^p_\omega( {Q_T}) \embedding f \in
L^{\tilde{q}}({Q_T})$ for any $\tilde{q} \le q_0$. Consequently, Lemma~\ref{lem:CZ} provides a weak solution $u$ to~\eqref{eq:sysL} such that
 \begin{equation}\label{wsqqq}
        \norm{\nabla u}_{L^{\tilde{q}}({Q_T})} \leq C(\tilde{A},q,\Omega)\norm{f}_{L^{\tilde{q}}({Q_T})},
 \end{equation}
 unique in this class. For further purposes let us also denote $Q_{\tilde{T}}:=(-1,T+1)\times \Omega$, extend $f$ by zero outside $Q_T$ and define $\tilde{A}(t,x):=\tilde{A}(T,x)$ for $t>T$ and $\tilde{A}(t,x):=\tilde{A}(0,x)$ for $t\le 0$. Then, it follows from Lemma~\ref{lem:CZ} that $u$ can be extended up to time $T+1$ and fulfils
\begin{equation}\label{wsq}
       \norm{\nabla u}_{L^{\tilde{q}}({Q_{\tilde{T}}})} \leq C(\tilde{A},q,\Omega)\norm{f}_{L^{\tilde{q}}({Q_T})},
\end{equation}
since we can extend $u$ by zero for $t\in (-1,0]$ so that it is a solution on the whole~$Q_{\tilde{T}}$.

Now, it suffices to prove the optimal estimate~\eqref{keyest}.  We divide
the rest of the proof into three steps. In the first one, we shall prove the
natural local interior estimates in $Q_{\tilde{T}}$, see Lemma~\ref{thm:ell-local}. Next, we obtain the local boundary estimates (Lemma~\ref{lem:halfball}) and finally, we combine them together to get~\eqref{keyest}. Lemmata~\ref{thm:ell-local} and~\ref{lem:halfball} may be of a independent interest.

\subsubsection{Interior estimates}
Recall that we write $z = (x, t)$. It is slightly more convenient in this section to use the following parabolic cylinders: $Q_r (z) = B_r (x) \times (t-r^2, t+r^2)$.
For a fixed $z_0 \in {Q_{\tilde{T}}} \setminus \partial {Q_{\tilde{T}}}$, we will call any parabolic cylinder  $Q_r (z_0)$ with $r \le 1$ an \emph{interior cylinder} as long as $Q_r (z_0) \subset \subset {Q_{\tilde{T}}}$. For any cylinder $Q=Q_r (z)$, we denote the coaxial cylinder $Q_{\alpha r} (z)$ by $\alpha Q$. The key interior estimate is formulated in the following lemma.
\begin{lemma}
\label{thm:ell-local}
Let  $p\in (1,\infty)$ and  $\omega \in \Acal_p$ be arbitrary. Assume that $Q_{2R}=B_{2R} \times I_{2R}$ is an interior cylinder with $R\le 1$, $f\in L^p_\omega(Q_{2R}; \setR^{n\times N})$, the strongly elliptic tensor\footnote{We recall here \eqref{str.elip} for the notion of strong ellipticity.} $\tilde{A}\in L^{\infty}(Q_{2R};\setR^{n\times N \times n \times N})$ and  $u\in  L^{\tilde{q}}(I_{2R}, W^{1,  \tilde{q}} (B_{2R};\mathbb{R}^N))$ with some $\tilde{q}>1$  satisfy for all $ \phi\in \mathcal{C}^{1}_0(Q_{2R})$ the following
\begin{equation}\label{eq:lind}
 \int_{{Q_{2R}}}   \big[- u (z) \,   \partial_t \phi (z) + \tilde{A}(z) \nabla u (z) \cdot \nabla \phi (z) - f (z) \cdot \nabla \phi (z) \big] \dz = 0.
\end{equation}
There exists $\delta>0$ depending only on $p$, ellipticity constants $c_1$, $c_2$ and $\Acal_p  (\omega)$ such that if
\begin{equation}\label{ass:AAA}
|\tilde{A}(z_1)-\tilde{A}(z_2)| \le \delta \quad \textrm{ for all } z_1,z_2\in {Q_{2R}}
\end{equation}
then the following interior local estimate holds
\begin{equation}\label{unlocal}
  \bigg(\dashint_{Q_R}\abs{\nabla u}^p\omega \dz\bigg)^\frac1p\leq
  C\bigg(  \dashint_{Q_{2R}}\abs{f}^p\omega \dz\bigg)^\frac1p
  + C \bigg(\dashint_{Q_{2R}}\omega \dz\bigg)^{\frac{1}{p}}\bigg(\dashint_{Q_{2R}}\abs{\nabla u}^{\tilde{q}} \dz\bigg)^{\frac{1}{\tilde{q}}},
\end{equation}
where 
 the constant $C$ depends only  on $p$, $c_1$, $c_2$ and $\Acal_p (\omega)$.
\end{lemma}
\begin{proof}[Proof of Lemma~\ref{thm:ell-local}]
Recall that for a weight $\omega$ and a set $S \subset \setR^{n+1}$ we write $\omega(S):=\int_S \omega \dz$. We can find and fix $q\in (1,\tilde{q})$ such that ${\frac{p}{q}} \ge \sigma$, where $\sigma>1$ is introduced in  Lemma~\ref{cor:leftopen}.  Therefore $\omega \in \Acal_{\frac{p}{q}}$. Note that $\nabla u\in L^{q}({Q_{2R}})$, because ${Q_{2R}}$ is bounded. Let us next introduce the centred maximal operator and the respective restricted maximal operator with power $q$
\[
 (M_q(g))(z):=\sup_{r>0}\bigg(\dashint_{Q_r(z)}\abs{g}^q \dy\bigg)^\frac1q,\quad (M_q^{<\rho}(g))(z):=\sup_{\rho\geq r>0}\bigg(\dashint_{Q_r(z)}\abs{g}^q dy \bigg)^\frac1q
\]
Since $M_q(g)=(M(\abs{g}^q))^\frac1q$ and $\omega\in \Acal_\frac{p}{q}$, the operator $M_q$ is bounded in $L^p_{\omega}(\setR^{n+1})$, see \eqref{defAp}. 

Having introduced the auxiliary notions, let us turn to the main part of our proof of \eqref{unlocal}. First, we set
\begin{equation}\label{def:Lambda}
  \Lambda:=\bigg(\dashint_{{Q_{2R}}} \abs{\nabla u}^q\bigg)^\frac1q.
\end{equation}
Thus for any $Q\subset \setR^{n+1}$ we immediately have
\begin{equation}
\label{eq:lamL}
 \bigg(\dashint_Q \abs{\chi_{{Q_{2R}}}\nabla u}^q \bigg)^\frac1q\leq \left(\frac{\abs{{Q_{2R}}}}{\abs{Q}}\right)^\frac1q\Lambda.
\end{equation}
Next, since the proof of \eqref{unlocal} will be based on the proper (`good-$\lambda$') estimates, we introduce the {open} level sets\footnote{The fact that the level sets are open follows from continuity of the maximal function.}
\begin{equation}\label{def:Olambda}
 O_\lambda:=\set{z\in \setR^{n+1}: \, M_q(\chi_{{Q_{R}}}\nabla u)(z)> \lambda}.
\end{equation}
We intend to use the Calder\'{o}n-Zygmund decomposition. Thus, for a fixed $\lambda \ge 2^{n+2} \Lambda$ and any $z\in {Q_{R}}\cap Q_{\lambda}$, using \eqref{eq:lamL}, the definition of $O_\lambda$ and continuity  of integrals with respect to the integration domain, we can find a cube $Q_{r_z}(z)$ such that
\begin{align}
\label{eq:cover1}
 \lambda^q< \dashint_{Q_{r_z}(z)}\abs{\chi_{Q_{R}}\nabla u}^q \leq 2\lambda^q\quad \text{ and } \quad\dashint_{Q_{r}(z)}\abs{\chi_{Q_{R}}\nabla u}^q  \leq 2\lambda^q \quad\text{ for all }r\geq r_{z}.
\end{align}
Moreover, using the estimate \eqref{eq:lamL}, the definition \eqref{eq:cover1} and the restriction imposed on $\lambda$, we have that
\[
  2^{n+2} \Lambda^q  \le 2^{(n+2)q} \Lambda^q \le \lambda^q < \dashint_{Q_{r_z}(z)}\abs{\chi_{Q_{R}}\nabla u}^q \le \frac{\abs{{Q_{R}}}}{\abs{Q_{r_z}(z)}} \Lambda^q = 2^{{n+2}} \frac{\abs{{Q_{R}}}}{\abs{2 Q_{r_z}(z)}} \Lambda^q.
\]
Consequently
\begin{equation}
\label{eq:cover1n}
\abs{2 Q_{r_z}(z)}\leq \abs{{Q_{R}}}.
\end{equation}

Next, using the Besicovich covering theorem, we extract a countable covering $\{Q_i\}_{i\in \mathbb{N}}$ of $O_{\lambda}$, where  $Q_i:=Q_{r_{z_i}}(z_i)$, such that the
$Q_i$'s have finite intersection, i.e. there exists a constant $C$ depending only on $n$ such that for all $i\in \mathbb{N}$
\begin{equation}\label{finite:prop}
\# \{j\in \mathbb{N}; \, Q_i\cap Q_j \neq \emptyset\} \le C.
\end{equation}
In addition, it follows from the construction that
\begin{align}\label{realn}
O_{\lambda}\cap {Q_{R}}=\bigcup_{i\in \mathbb{N}} (Q_i\cap {Q_R}).
\end{align}

Using the fact that $Q_i=Q_{r_i}(z_i)$ with some $z_i\in {Q_{R}}$ and \eqref{eq:cover1n}, we observe that  $2Q_i\subset {Q_{2R}}$ and for a constant $C$ depending only on the dimension $n$
\begin{align}
\label{eq:cover2}
\abs  {Q_i}\le C(n)\abs{Q_i\cap {Q_R}}.
\end{align}
Since $\omega\in \mathcal{A}_p$ the above relation implies
(see e.g. Stein~\cite{Ste93}, \S V.1.7)
\begin{align}
  \label{eq:cover2omega}
  \omega(Q_i)\le C(n, A_p(\omega))\, \omega(Q_i\cap Q_R).
\end{align}

Next, we are going to use the re-distributional estimates with respect to the right hand side. To this end  for an arbitrary $\epsilon>0$ and $k\ge 1$, we introduce the re-distributional set
\[
 U^\lambda_{\epsilon, k}:=O_{k\lambda}\cap\set{z\in \setR^{d}:\, M_q(f \chi_{Q_{2R}})(z)\leq \epsilon\lambda},
\]
where $f$ is given in \eqref{eq:lind}, i.e., it is the right hand side of our problem.
Finally, let us assume for a moment that the  following statement holds true (here $\delta$ comes from our assumptions on tensor $\tilde{A}$, see \eqref{ass:AAA}):
\begin{equation}
\begin{aligned}
 \label{eq:redis}
 &\textrm{There exists $k\ge 1$ depending only on $c_1$, $c_2$, $d$, $p$, $A_p(\omega)$ such that for all $\varepsilon\in  (0,1)$}\\
  &\textrm{and all $\lambda\ge 2^{n+2}\Lambda\;$ it holds }\quad
\abs{Q_i\cap U^\lambda_{\epsilon, k}\cap Q_R}\leq C(c_1,c_2,n)(\epsilon +\delta)\abs  {Q_i}.
\end{aligned}
\end{equation}

We continue for now with the proof and postpone justifying \eqref{eq:redis} to the end. Using  the  H\"older inequality, the reverse H\"older inequality
for $\Acal_p$-weights (compare \eqref{reversom}), \eqref{eq:redis}
and~\eqref{eq:cover2omega}, we obtain for some $r>1$ depending only on
$d$, $p$ and $A_p(\omega)$ 
\begin{align*}
  \omega(Q_i \cap U_{\epsilon, k}^\lambda \cap {Q_R}) &\leq  \abs  { Q_i}
  \bigg(\dashint_  { Q_i}\omega^r
 \bigg)^\frac1r\left(\frac{\abs{Q_i\cap
        U^\lambda_{\epsilon,k}\cap {Q_R}}}{\abs  { Q_i}}\right)^\frac1{r'}
  \\
  &\leq C(d,p,A_p(\omega),c_1,c_2)(\epsilon+\delta)^\frac1{r'}
  \omega(Q_i)\\
  &\leq  C(d,p,A_p(\omega),c_1,c_2)(\epsilon+\delta)^\frac1{r'}
  \omega(Q_i \cap {Q_R}).
\end{align*}
Consequently, using the subadditivity of $\omega$ and the finite intersection property of $Q_i$, i.e., the estimate \eqref{finite:prop}, we find
\begin{align}
\label{eq:smallness}
\omega(U_{\epsilon, k}^\lambda \cap {Q_R})\leq C(d,A_p(\omega),c_1,c_2)(\epsilon+\delta)^\frac1{r'}\omega(O_\lambda\cap {Q_R}),
\end{align}
which is the essential estimate for what follows.

Finally, using the Cavalieri principle (the Fubini theorem), we obtain
\begin{align}\label{sce}
\begin{aligned}
 \int_{{Q_R}}\abs{\nabla u}^p \omega \dz &= p\int_0^\infty \omega(\set{(\nabla u)\chi_{{{Q_R}}}>\lambda})\lambda^{p-1} {\rm{d}}\lambda\\
 &\leq C\Lambda^p\omega({{Q_R}}) + p\int_{k 2^{n+2}\Lambda}^\infty\lambda^{p-1}\omega(O_\lambda \cap {{Q_R}})\rm{d}\lambda.
\end{aligned}
\end{align}
Therefore, to get the estimate \eqref{unlocal}, we need to estimate the last term on the right hand side. To do so let us use the definition of $U^{\lambda}_{\epsilon, k}$ and a change of variables to start with the following estimate valid  for all $m> k 2^{n+2}\Lambda$ (note here that the integration domain is chosen such that $\lambda/k >2^{n+2}\Lambda$ so that we can use \eqref{eq:smallness} below)
\begin{align*}
& \int_{k2^{n+2}\Lambda}^m\lambda^{p-1}\omega(O_\lambda \cap {{Q_R}})\rm{d}\lambda\\
&\leq \int_{k 2^{n+2}\Lambda}^m\lambda^{p-1}\omega(U^{\frac{\lambda}{k}}_{\epsilon, k}\cap {{Q_R}}){\rm{d}}\lambda
+ \int_{k 2^{n+2}\Lambda}^m\lambda^{p-1}\omega\Big(\{M_q( f\chi_{Q_{2R}})>\frac{\epsilon\lambda}{k}\}\Big){\rm{d}}\lambda\\
&\overset{\eqref{eq:smallness}}\leq C(\epsilon+\delta)^\frac1{r'}\int_{k 2^{n+2}\Lambda}^m\lambda^{p-1}\omega(O_\frac{\lambda}{k}\cap {Q_R}){\rm{d}}\lambda+ \frac{k^p}{p\epsilon^p}\int_{\setR^n}\abs{M_q(f \chi_{{Q_{2R}}})}^p \omega (z) \dz\\
&= C k^{p}(\epsilon+\delta)^\frac1{r'}\int_{2^{n+2}\Lambda}^\frac{m}{k}\lambda^{p-1}\omega(O_\lambda\cap {Q_R}){\rm{d}}\lambda +   \frac{k^p}{p\epsilon^p}\int_{\setR^n}\abs{M(f^q \chi_{{Q_{2R}}})}^{p/q} \omega (z) dz \\
&\le Ck^{p}(\epsilon+\delta)^\frac1{r'}\int_{2^{n+2}\Lambda}^{k 2^{n+2}\Lambda}\lambda^{p-1}\omega(O_\lambda\cap {Q_R}){\rm{d}}\lambda+
Ck^{p}(\epsilon+\delta)^\frac1{r'}\int_{k 2^{n+2}\Lambda}^m\lambda^{p-1}\omega(O_\lambda\cap {Q_R}){\rm{d}}\lambda \\
&\quad + C_1  \frac{k^p}{p\epsilon^p} \int_{Q_{2R}}\abs{f}^p \omega \dz,
\end{align*}
where for the last inequality we have used the fact that $\omega \in \Acal_{\frac{p}{q}}$, the related strong property of maximal function and $C_1 = C(d, p,c_1, c_2,A_p(\omega))$. Observe that $k$ is already fixed by \eqref{eq:redis}. At this point, we fix the maximal value of $\delta$ arising in the assumption of Lemma~\ref{thm:ell-local}. Namely, we set  $\epsilon:=\delta$ and chose $\delta$ such that $Ck^{p}(2 \delta)^\frac1{r'}\le\frac12$. Consequently, we can  absorb the middle term of the final inequality above into the left hand side and letting $m\to \infty$, we find that
\begin{align*}
 \int_{k 2^{n+3}\Lambda}^{\infty}\lambda^{p-1}\omega(O_\lambda \cap {Q_R})\rm{d}\lambda
&\le C(k,p,q,A_p(\omega))\left(\int_{Q_{2R}}\abs{f}^p \omega \dz+ \Lambda^p \omega({Q_R})\right).
\end{align*}
Using this in \eqref{sce}, recalling the definition of $\Lambda$ (see \eqref{def:Lambda}) and via $q \le \tilde q$, we find \eqref{unlocal}.

To finish the proof, it remains to validate  \eqref{eq:redis}. To this end assume that $Q_i\cap {Q_R} \cap U^\lambda_{\epsilon, k}\neq\emptyset$. 
So,  taking $z \in Q_i \cap U^\lambda_{\epsilon, k}$ via  the definition of $U^{\lambda}_{\epsilon,k}$, for any $r>0$
\[
\dashint_{Q_r (z)}\abs{f (y)}^q \dy \leq \epsilon^q \lambda^q.
\]
In particular, it holds also for $r= 4r_{z_i}$ and since $Q_{4r_{z_i}} (z) \supset 2Q_i$, we get
\begin{align}\label{smalfe}
 \bigg(\dashint_{2Q_i}\abs{f}^q \dz\bigg)^\frac1q\leq 2^{n+2} \epsilon \lambda.
\end{align}
Let us now freeze coefficients of the tensor $\tilde{A}(z)$ in the centre $z_i$  of the cube $Q_i$, writing  $\tilde{A}_i =  \tilde{A}(z_i)$ and consider  the following constant coefficient problem (that will serve as a comparison problem)
$$
\begin{aligned}
\partial_t w -\divergence \tilde{A}_i\nabla w&=\divergence(( \tilde{A}-\tilde{A}_i) \nabla u-f) &&\textrm{in }2Q_i,\\
w&=0 &&\textrm{on }\partial 2Q_i.
\end{aligned}
$$
Lemma~\ref{lem:CZ} yields existence of a solution to the above problem as well as the following estimate
\begin{align}
\label{eq:comp}
 \dashint_{2Q_i}\abs{\nabla w}^q\, \dz \leq C\dashint_{2Q_i}\abs{\tilde{A}-\tilde{A}_i}^q\abs{\nabla u}^q \dz + C\dashint_{2Q_i}\abs{f}^q \dz \leq C(\epsilon^q +\delta^q)\lambda^q,
\end{align}
where for the second inequality we used \eqref{eq:cover1} with \eqref{eq:cover1n} (implying   $2Q_i\subset {Q_{2R}}$), \eqref{smalfe} and the assumed  $|\tilde{A}(z_1)-\tilde{A}(z_2)|\le \delta$ for all $z_1,z_2 \in {Q_{2R
}}$. Furthermore, since $u$ satisfies \eqref{eq:lind}, the difference  $h= u-w \in L^q (W^{1,q})$ fulfils in the sense of distributions
\begin{equation}
 \label{eq:har}
\partial_t h -\divergence (\tilde{A}_i\nabla h) = 0\quad \text{ in }\quad2Q_i.
\end{equation}
Since it is a constant coefficient strongly parabolic problem, $h$ is locally smooth (for instance via localisation that produces lower-order right hand side, classical regularity theory for initial-boundary value problems and bootstrapping this step on smaller cubes). Hence, we can differentiate it and obtain the following estimate (again via localisation, regularity theory for initial-boundary value problems, embeddings and bootstrapping)
\begin{align}
  \label{eq:harnack}
  \sup_{\frac32 Q_i}\abs{\nabla h}\leq
  C \bigg( \dashint_{2Q_i}\abs{\nabla h}^q \dz \bigg)^\frac1q
\end{align}
where the constant $C$ depends only on $n$, $c_1$ and $c_2$.

Next, for any  $z\in Q_i$ and $r>r_i /2$, we have that $Q_{r}(z)\subset Q_{3r}(z_i)$. 
Consequently, it follows from \eqref{eq:cover1} that
\[
 \dashint_{Q_r(z)}\abs{\chi_{{Q_{R}}} \nabla u}^q \dy \leq 3^{n+2} \dashint_{Q_{3r}(z_i)}\abs{\chi_{{Q_{R}}}\nabla u}^q \dy  \leq 3^{n+3}\lambda^q.
\]
Therefore, assuming that  $k\geq 3^{n+3}$, we obtain
 that for all $z\in Q_i\cap \set{M_q(\nabla u \chi_{{Q_R}} )>k\lambda}$
\[
 M_q(\nabla u \chi_{{Q_R}} )(z)=M_q^{<\frac{r_i}{2}}(\nabla u \chi_{{Q_R}} )(z).
 \]
This identity, the sublinearity of the maximal operator, $u= w+h$ and \eqref{eq:harnack} imply  that for all $z\in Q_i\cap \set{M_q(\nabla u \chi_{{Q_R}} )>k\lambda}$
\begin{align*}
 M_q(\nabla u \chi_{{Q_R}} )(z)&=M_q^{<\frac{r_i}{2}}(\nabla u \chi_{{Q_R}} )(z)
\leq M_q^{<\frac{r_i}{2}}(\nabla h)(z)+M_q^{<\frac{r_i}{2}}(\nabla w)(z)\\
 &\leq  C\bigg(\dashint_{2Q_i}\abs{\nabla h}^q \dy\bigg)^\frac1q +M_q^{<\frac{r_i}{2}}(\nabla w)(z).
\end{align*}
Finally, using the triangle inequality (applied for $h=u-w$) and  the estimates \eqref{eq:comp} with $(\epsilon^q +\delta^q) \le 1$ as well as \eqref{eq:cover1}, we conclude
\begin{align*}
M_q(\nabla u \chi_{{Q_R}} )(z) \leq C\lambda+M_q^{<\frac{r_i}{2}}(\nabla w)(z).
\end{align*}
Hence, setting $k:=\max\set{C+1,3^{n+3}}$, we obtain
\[
\begin{aligned}
Q_i\cap {Q_R} \cap U^\lambda_{\epsilon, k} = Q_i\cap {Q_R} \cap
\set{M_q(\nabla u \chi_{{Q_{R}}})> k\lambda} \cap\set{ M_q(f \chi_{Q_{2R}})\leq \epsilon\lambda}  \subset \\
 Q_i\cap {Q_R} \cap  \set{z \in Q_i \; |\; M_q^{<\frac{r_i}{2}}(\nabla w)\geq \lambda} \cap\set{ M_q(f \chi_{Q_{2R}})\leq \epsilon\lambda} .
 \end{aligned}
\]
Since the  restricted maximal operator above does not see the values outside $2 Q_i$, we can extend $\nabla w$ by zero to the whole $\mathbb{R}^{n+1}$, invoke the weak estimate for the maximal functions and finally  the estimate \eqref{eq:comp} to conclude from the above inclusion
\[
 \abs{Q_i\cap {Q_R} \cap U^\lambda_{\epsilon, k} } \leq \abs{ \set{ M_q^{<\frac{r_i}{2}}(\chi_{2Q_{i}}\nabla w)\geq \lambda}} \leq \frac{C}{\lambda^q} \int_{2Q_i}\abs{\nabla w}^q \dz \leq C(\epsilon^q+\delta^q)\abs{Q_i}, \]
which finishes the proof of \eqref{eq:redis} and hence of Lemma~\ref{thm:ell-local}.
\end{proof}

\subsubsection{Estimates near the boundary.}
In this subsection, we derive the estimates on parabolic cylinders near the boundary. Here, we also use the extension of the solution to the time interval $(-1, T+1)$ and consider the domain $Q_{\tilde{T}}$. Thus, we are interested only in the behaviour of the solution near $\partial \Omega$. Hence, we say that $Q_R(z_0)$ is a \emph{boundary cylinder} if $z_0\in \partial \Omega \times [0,T]$.
The main result of this subsection is the following lemma.
\begin{lemma}
\label{lem:halfball}
Let $\partial \Omega \in \mathcal{C}^1$,  $p\in (1,\infty)$,   $\omega \in \Acal_p$ and $\tilde{A}\in L^{\infty}(Q_{\tilde{T}};\setR^{n\times N \times n \times N})$ be a strongly elliptic tensor. Then there exist $R_0>0$  and $\delta>0$ depending only on $p$ $\Omega$, $\Acal_p  (\omega)$ and ellipticity constants $c_1$, $c_2$,  such that if $Q_R$ with $R\in (0,R_0)$ is a boundary cylinder and
\begin{equation}\label{spb}
|\tilde{A}(z_1)-\tilde{A}(z_2)| \le \delta \quad \textrm{ for all } z_1,z_2\in {Q_{2R} \cap Q_{\tilde{T}}},
\end{equation}
then for any $u\in L^{\tilde{q}}(Q_{\tilde{T}}\cap Q_{2R}; \mathbb{R}^N)$ with $\nabla u\in  L^{\tilde{q}}(Q_{\tilde{T}}\cap Q_{2R}; \mathbb{R}^{N\times n})$ and $u=0$ on $((-1,T+1)\times \partial \Omega)\cap Q_R$ and any  $f\in L^p_\omega((Q_{\tilde{T}}\cap Q_{2R}; \mathbb{R}^N); \setR^{n\times N})$  satisfying  for any $ \phi\in \mathcal{C}^{1}_0(Q_{\tilde{T}}\cap Q_{2R}; \mathbb{R}^N))$
\begin{equation}\label{wfL2}
 \int_{{Q_{\tilde{T}}}}  - u (z) \,   \partial_t \phi (z) \dz + \tilde{A}(z) \nabla u (z) \cdot \nabla \phi (z) \dz - f (z) \cdot \nabla \phi (z) \dz = 0,
\end{equation}
the following boundary local estimate holds
\begin{equation}\label{uptobound}
  \bigg(\dashint_{Q_R \cap {Q_{\tilde{T}}}}\abs{\nabla u}^p\omega \dz\bigg)^\frac1p\leq
  C\bigg(  \dashint_{Q_{2R} \cap {Q_{\tilde{T}}}}\abs{f}^p\omega \dz\bigg)^\frac1p
  + C \bigg(\dashint_{Q_{2R} \cap {Q_{\tilde{T}}}}\omega \dz\bigg)^{\frac{1}{p}}\bigg(
  \dashint_{Q_{2R} \cap {Q_{\tilde{T}}}}\abs{\nabla u}^{\tilde{q}} \dz\bigg)^{\frac{1}{\tilde{q}}},
\end{equation}
with the constant $C$ depending only  on $\Omega$, $p$, $c_1$, $c_2$ and $\Acal_p (\omega)$.
\end{lemma}

\begin{proof}
Briefly, the strategy is to straighten locally the boundary and via the null extension in time and an odd reflection in space to reduce the boundary case to the interior case of the previous lemma. The details follow.

Since $\partial \Omega\in \mathcal{C}^1$, we know that for any $\varepsilon>0$ we can find $R_0>0$ such that we can locally change the coordinates $y =\Psi (x), x \in B_{2R_0}(x_0)$ (translation and rotation) and we are allowed  to write
\[
 \Psi (\partial \Omega \cap B_{2R_0}(x_0)) = \{(y',y_n) :  \, |y'|<\alpha, \, a(y')=y_n\}
\]
with $a \in \mathcal{C}^1([-\alpha,\alpha]^{n-1}; \mathbb{R})$
, where $y' = y_1, \dots y_{n-1}$ and  that
\begin{equation}\label{small_pert}
\sup_{|y'|<\alpha} |a(y')|+\epsilon |\nabla a (y')| \le \epsilon^2.
\end{equation}
For brevity, from now on let us write $B:=B_{R}(x_0)$ with an arbitrary $R\le R_0$ and $I$ for the time interval of $Q_R=B \times I$. In addition, we can also assume that $R<1/2$ and consequently $2I\subset [-1,T+1]$. Let us introduce the related curvilinear interior and exterior half-cubes
$$
\begin{aligned}
H^+&:=\{(y',y_n):\, |y'|<\alpha, \; a(y')-\beta< y_n <a(y')\}\supset \Psi (\Omega \cap 2B),\\
H^{-}&:=\{(y',y_n):\, |y'|<\alpha, \; a(y')< y_n <a(y')+\beta\}\supset \Psi (\Omega^c \cap 2B),
\end{aligned}
$$
where the inclusions follow from choice of $\beta$ and regularity of $\Omega$. Let us observe that due to $N$ being finite, $\alpha$ and $\beta$ can be treated as fixed. The related sets in the original $x$ variables are
$$
\begin{aligned}
\Omega_0^+:= \Psi^{-1} (H^+) \supset \Omega \cap 2B,\quad 
\Omega_0^{-}:= \Psi^{-1} (H^-) \supset \Omega^c \cap 2B.
\end{aligned}
$$
Next, on the level of variables $y$, let us define the curvilinear reflection $R: H^+ \to H^-$ as
$$
R(y',y_n):=(y',2a(y')-y_n) 
$$
Observe that $|\det \nabla_y R|\equiv 1$ and $R$ and $R^{-1}$ are $\mathcal{C}^1$ mappings. Consequently, also
\[
T := \Psi^{-1} \circ R \circ \Psi: \Omega_0^+ \to \Omega_0^-
\]
satisfies $|\det J|\equiv 1$, where we defined $J:=\nabla_x T$. Mappings $T$ and $T^{-1}$ are $\mathcal{C}^1$.

Finally, let us extend all quantities into $\Omega_0^-$ as follows:
\begin{align*}
\tilde{u}(t,x)&:=\left\{\begin{aligned}
&u(t,x) &&\textrm{for } x\in \Omega_0^+,\\
&-u(t, T^{-1}(x)) &&\textrm{for } x\in \Omega_0^{-},
\end{aligned}
\right.  \\
\tilde{A}(t,x)&:=\left\{\begin{aligned}
&\tilde{A}(t,x) &&\textrm{for } x\in \Omega_0^+,\\
&(J(T^{-1}x) \otimes J(T^{-1}x)) A(t, T^{-1}x)&&\textrm{for } x\in \Omega_0^{-},
\end{aligned}
\right.  \\
\tilde{f}(t,x)&:=\left\{\begin{aligned}
&f(t,x) &&\textrm{for } x\in \Omega_0^+,\\
&-J(T^{-1}x)f(t, T^{-1}(x)) &&\textrm{for } x\in \Omega_0^{-},
\end{aligned}
\right.  \\
\tilde{\omega}(t,x)&:=\left\{\begin{aligned}
&\omega(t,x) &&\textrm{for } x\in \Omega_0^+,\\
&\omega(t, T^{-1}(x)) &&\textrm{for } x\in \Omega_0^{-},
\end{aligned}
\right.
\end{align*}
where $\otimes$ denotes the outer product of two matrices (a tensor).
Let us also introduce
$$
M :=  \Omega_0^+ \cup \Omega_0^- \cup  \partial \Omega_0 \supset 2B.
$$
Since $u$ has zero trace on $\partial \Omega$, we see that
$\tilde u \in  L^{\tilde{q}}({2I}, W^{1,  \tilde{q}} (M))$.


Our aim is now to show that for any $\varphi \in \mathcal{C}^{1}_0(2B \times 2I)$  the following identity holds
\begin{equation}
 I := \int_{{M \times 2 I } }  \big[ - \tilde u  \,   \partial_t \phi   + \tilde A \nabla \tilde u  \cdot \nabla \phi  - \tilde f  \cdot \nabla \phi  \big] \dx dt = 0. \label{extend}
\end{equation}

First, let us take any $\eta \in \mathcal{C}^{1}(\Omega_0^+ \times 2 I)$ and define $\tilde \eta \in \mathcal{C}^{1}(\Omega_0^- \times 2 I)$ as
$ \tilde \eta (t,x):= \eta (t, T^{-1} (x)) $. We observe, via a variable change $T( \Omega_0^+) = \Omega_0^-$, a straightforward computation and our definitions of the respective extensions, that
\begin{equation} \label{eq:mirror}
\intop_{\Omega_0^- \times (2 I)}   \!\!\!\!\!\! [- \tilde u \,   \partial_t  \tilde \eta + \tilde A \nabla \tilde u \cdot \nabla  \tilde \eta - \tilde f \cdot \nabla  \tilde \eta ] \dx dt =
- \!\!\!\!\!\!\!\!  \intop_{\Omega_0^+ \times (2 I)}   \!\!\!\!\!\! [- u \,   \partial_t  \eta +  \tilde{A} \nabla  u \cdot \nabla \eta -  f \cdot \nabla   \eta ] \dx dt .
\end{equation}
It is important  to notice that $\eta$ and $\tilde \eta$ may not vanish on the boundary, in particular on $\partial\Omega_0$, since the relation \eqref{eq:mirror} is just a variable change (not PDE related).
Next let us take an arbitrary $\varphi \in \mathcal{C}^{1}_0(2B \times 2I)$ and define its symmetrisation as
\begin{align*}
  \overline{\phi} (t,x) &:=
  \begin{cases}
      \phi (t,x) &\qquad \text{for } x \in \Omega_0^+, \\
   \tilde \phi (t,x) = \phi (t, T^{-1} (x) ) &\qquad \text{for } x \in \Omega_0^-,
  \end{cases}
\end{align*}
then $\overline{\phi} \in \mathcal{C}^{1}_0(2B \times 2I)$
and~\eqref{eq:mirror} implies
\begin{align*}
 \int_{M \times  (2 I)}  \big[ - \tilde u  \, \partial_t   \overline \phi   + \tilde A \nabla \tilde u  \cdot \nabla  \overline \phi  - \tilde f  \cdot \nabla  \overline \phi  \big] \dx dt = 0
\end{align*}
Subtracting the above $0$ from the l.h.s. of \eqref{extend} and next observing that the resulting test function $\phi - \overline{\phi} \in \mathcal{C}^{1}_0(2B \times 2I)$ vanishes on $\Omega_0^+ \times 2 I$ by the definition of $ \overline{\phi}$, we have
\[
I =  \int_{\Omega_0^- \times (2 I)}  \big[ - \tilde u  \,   \partial_t (\phi - \overline{\phi} )  + \tilde A \nabla \tilde u  \cdot \nabla (\phi - \overline{\phi})  - \tilde f  \cdot \nabla (\phi - \overline{\phi})  \big] \dx dt
\]
We rewrite the r.h.s. above using  \eqref{eq:mirror} with $\tilde \eta = (\phi - \overline{\phi})_{| \Omega_0^- \times 2 I}$. Hence the above equality takes the form, with $\hat \phi (t,x) := (\phi - \overline{\phi}) (t, T^{-1} (x))$
\[
I =  \int_{\Omega_0^+ \times (2 I \cap [0,T])}  \big[ -  u  \,   \partial_t \hat \phi  +  A \nabla u  \cdot \nabla \hat \phi   - f  \cdot \nabla \hat \phi   \big] \dx dt;
\]
thanks to its definition $\hat \phi \in \mathcal{C}^{1}_0 \left((\overline{\Omega_0^+} \cap 2B) \times 2 I \right)$,  so\footnote{Actually, $\hat \phi (t)$ is defined on $(\Omega_0^+ \cup \Omega_0^-) \cap 2B$, whose interface is $\partial \Omega_0 \cap 2B$ and $\hat \phi (t) $ vanishes on $\Omega_0^- \cap 2B$, hence it can be extended by zero to $\partial \Omega_0 \cap 2B$.} 
\[
I = \int_{{Q_T}}  \big[ -  u  \,   \partial_t \hat \phi  +  A \nabla u  \cdot \nabla \hat \phi   - f  \cdot \nabla \hat \phi   \big] \dx dt
\]
for a $\hat \phi \in \mathcal{C}^{1}_0 \left((\overline{\Omega_0^+}  \cap 2B) \times 2 I \right)$.
This and admissibility of $\hat \phi$ into \eqref{wfL2}
proves~\eqref{extend}, i.e.
\begin{equation}\label{extend2}
 \int_{{Q_T} }  \big[ - \tilde u  \,   \partial_t \phi   + \tilde A \nabla \tilde u  \cdot \nabla \phi  - \tilde f  \cdot \nabla \phi  \big] \dx dt = 0
\end{equation}
for any $\varphi \in \mathcal{C}^{1}_0(2B \times 2I)$ (As before, the difference in integration domain between \eqref{extend} and ~\eqref{extend2} is mitigated by a support of the test function). Consequently, \eqref{extend2} suggests an application of the interior Lemma~\ref{thm:ell-local}.

Hence to conclude, we need to check if $\tilde u, \tilde A, \tilde f$ and $\tilde \omega$ satisfy the assumptions of Lemma~\ref{thm:ell-local}. Regularity classes of $\tilde u, \tilde A, \tilde f$ and $\tilde \omega$ follows from their definitions and change of variables. Since $J$ is the Jacobian of a product of a small perturbation of the odd reflection matrix $R$, recall \eqref{small_pert} (hence composition of two $R$'s, present in the definition of $\tilde{A}$, is a small perturbation of identity)  and a matrix of translation and rotation $\Psi$, strong ellipticity of $A$ implies strong ellipticity of  $\tilde{A}$,  for an appropriately chosen $\epsilon$ of \eqref{small_pert} in relation to the ellipticity constants $c_1, c_2$ of $A$. The  ellipticity constants of $\tilde{A}$ depend thus on $c_1, c_2$ and  the shape of $\Omega$. (The choice of $\epsilon$ here and consequently of $R_0$  also influences the upper bound on diameter of $2B$.)

 Finally we need to show small oscillations of $\tilde{A}$. Using boundedness of $A$ and $J$ we have
$$
\begin{aligned}
\sup_{t,s \in  2I} \sup_{x,y \in  \Omega_0^-}|\tilde{A}(t, x)-\tilde{A}(s, y)|&\le \sup_{t,s \in  2I} \sup_{x,y \in \Omega_0^+}|J(x)A(t, x)J^T(x)-J(y)A(s, y)J^T(y)|\\
&\le C \sup_{t,s \in  2I} \sup_{x,y \in \Omega_0^+} |A(t, x)-A(s, y)| + C \sup_{x,y \in \Omega_0^+} |J(x)-J(y)| \\
&\le C \delta + C \epsilon,
\end{aligned}
$$
with the last inequality following from our assumption \eqref{spb} and  \eqref{small_pert}. A similar computation for other cases implies that
\[
\sup_{t,s \in  2I} \sup_{x,y \in 2B}|\tilde{A}(t, x)-\tilde{A}(s, y)| \le C \delta + C \epsilon.
\]
This allows us to choose $R_0$ so small and consequently $\epsilon$ so small that the small oscillation assumption of Lemma~\ref{thm:ell-local} is satisfied. Observe that this choice and the previous assumption that $2B$ intersects a single portion of boundary $\partial \Omega_0$ is an upper bound on the diameter.

Now directly from Lemma \ref{thm:ell-local} and the $T$-related variable change we have for any boundary cylinder $Q_R$ with $R\le R_0$ (note that $R_0$ is already fixed)
\[
\begin{aligned}
& \bigg(\frac{1}{|Q_R|} \int_{Q_R \cap {Q_T}}\abs{\nabla u}^p\omega \dz\bigg)^\frac1p\leq \\
& C\bigg(  \frac{1}{|Q_{2R}|} \int_{Q_{2R}  \cap {Q_{\tilde T}}}\abs{f}^p\omega \dz\bigg)^\frac1p
  + C \bigg(\frac{1}{|Q_{2R}|} \int_{Q_{2R}  \cap {Q_{\tilde{T}}}}\omega \dz\bigg)^{\frac{1}{p}}\bigg(\frac{1}{|Q_{2R}|} \int_{Q_{2R}  \cap {Q_{\tilde{T}}}}\abs{\nabla u}^{\tilde{q}} \dz\bigg)^{\frac{1}{\tilde{q}}}.
  \end{aligned}
\]
Due to the assumed regularity $\Omega \in \mathcal{C}^1$ and local `flatness', compare \eqref{small_pert}, we can replace measures of $Q_R$ and $Q_{2R}$ by the desired ones, at a cost of changing $C$, hence obtaining \eqref{uptobound}.
\end{proof}

\subsubsection{Global optimal estimate  \eqref{keyest}}
Now we will combine local interior (Lemma~\ref{thm:ell-local}) and boundary (Lemma~\ref{lem:halfball}) estimates into the optimal global estimate  \eqref{keyest}, thus completing the proof of Theorem~ \ref{thm:lin} for homogenous initial-boundary data.

Let us recall that we have the weak solution $u$ to  \eqref{eq:sysL} with $f \in L_\omega^{p}({Q_T})$  such that
  \begin{equation}\label{eq:subo}
 \norm{\nabla u}_{L^{\tilde{q}}({Q_{\tilde{T}}})} \leq C(\tilde{A},q,\Omega, {A}_p(\omega))\norm{f}_{L_\omega^{p}({Q_T})}.
 \end{equation}
for a $\tilde{q} >1$ related to \eqref{eq:lqprop}, compare \eqref{wsq}.
Such $u$ satisfies the respective assumption of Lemmata  \ref{thm:ell-local}, \ref{lem:halfball}.
    Let us fix $\delta>0$ and $R_0$ (and consequently $\epsilon$) in accordance with Lemmata  \ref{thm:ell-local}, \ref{lem:halfball}.
  Since~$\Omega$
has~$\mathcal{C}^1$ boundary, we can find a finite covering $\cup_{i=1}^N Q_i$  of ${Q_T}$ by parabolic interior and boundary cylinders and due to the continuity of $\tilde{A}$ we have all assumptions of  Lemmata  \ref{thm:ell-local} and \ref{lem:halfball} satisfied, i.e.,
$$
\sup_{z_1,z_2\in {Q_T}  \cap 2 Q_i}|\tilde{A}(z_1)-\tilde{A}(z_2)|\le \delta.
$$

Now it follows from \eqref{unlocal} and \eqref{uptobound} that
\begin{equation*}
\int_{{Q_T}}\abs{\nabla u}^p\omega \dx \leq
  C \int_{{Q_T}}\abs{f}^p\omega \dx+   C \sum_i
  \frac{\omega(2Q_i \cap {Q_T})}{|2Q_i \cap {Q_T}|^{\frac{p}{\tilde{q}}}}
  \left(\int_{{Q_T}}\abs{\nabla u}^{\tilde{q}}
    \dx\right)^{\frac{p}{\tilde{q}}} \leq   C \int_{{Q_T}}\abs{f}^p\omega \dx
\end{equation*}
with $C= C(A,\Omega, A_p(\omega), N)$, where the second inequality follows from \eqref{eq:subo} and finiteness of the involved sum.
 This finishes the proof of Theorem~\ref{thm:lin} in the case of homogenous initial boundary data $\tilde{g} \equiv 0$.

\subsection{Inhomogenous initial boundary data}
Let us take in already proven homogenous version of Theorem~\ref{thm:lin} right hand side (force) $H := f+ F -\tilde A \nabla g$. We obtain solution $v$ that satisfies
\begin{equation}\label{wf:lin_in}
\int_{Q_T} -v\cdot \partial_t \varphi + \tilde{A} \nabla v \cdot \nabla \varphi \dz = \int_{Q_T} (f+F- \tilde A \nabla g)\cdot \nabla \varphi \dz.
\end{equation}
with estimate
$$
\int_{Q_T} |\nabla v|^p \omega \dz \le C\int_{Q_T}(|f|^p + |F|^p +|\nabla g|^p)\omega \dz.
$$
implying for $u:=v+g$ via the triangle inequality that \eqref{keyest}. Identity \eqref{wf:lin_in} is \eqref{wf:lin}.

Concerning the uniqueness, we see that
\begin{equation*}
\int_{Q_T} -(u_1-u_2 -g_1+g_2)\cdot \partial_t \varphi + \tilde{A} \nabla (u_1-u_2) \cdot \nabla \varphi \dz = \int_{Q_T} (F_1-F_2)\cdot \nabla \varphi \dz
\end{equation*}
for all $\varphi \in \mathcal{C}_0^1((-\infty,T)\times \Omega)$. Due to the compatibility assumption \eqref{uniq:nonlin}, it however follows that
\begin{equation*}
\int_{Q_T} -(u_1-u_2)\cdot \partial_t \varphi + \tilde{A} \nabla (u_1-u_2) \cdot \nabla \varphi \dz =0.
\end{equation*}
Consequently, since according to Lemma~\ref{lem:CZ} the only solution for zero data is zero and since $u_1=u_2$ on $\partial \Omega \times (0,T)$, we see that $u_1=u_2$ almost everywhere in $Q_T$. Hence, the proof of Theorem~\ref{thm:lin} is complete.


\section{Proof of Theorem \ref{th}}
\label{sec:pf}

As in the linear case, let us for now consider the homogenous case, i.e., the case  $g \equiv 0$ and $F\equiv 0$. The way to recover the inhomogenous case will be sketched at the end of this section.

Let us take in \eqref{eq:sysA} an arbitrary fixed forcing term  $f\in L_\omega^q ({Q_T};\mathbb{R}^{n\times N})$ with a $q \in (1,\infty)$ and with an arbitrary $\omega \in \mathcal{A}_q$. Then we know that there is a $q_0\in (1,\min\{2,q\})$, such that $f\in L^{q_0}(Q_T; \mathbb{R}^{n\times N})$. 
Defining $\omega_0 := (1+ {M}f)^{q_0-2}$, we can use Lemma~\ref{cor:dual} to obtain that $\omega_0 \in \mathcal{A}_2$ and it is evident that  $f \in L^2_{\omega_0}({Q_T}; \setR^{n\times N})$. With this basic notation, we show the existence of a weak solution fulfilling \eqref{eq:WL}.

\subsection{Approximative problems}

We set $f^k:=f \chi_{\set{|f|<k}}$. It is evident that $f^k$ are bounded functions, $|f^k|\nearrow |f|$ and
\begin{align}\label{cfn}
  f^k \to f &&\textrm{strongly in } L^2_{\omega_0}\cap
  L_\omega^{q}({Q_T}; \setR^{n\times N}).
\end{align}
Therefore, for any $k$, we can apply the standard monotone operator theory to find a weak solution to
\begin{align}
\label{eq:aprox}
\begin{aligned}
  \partial_tu^k-\divergence A(z, \nabla u^k) &= -\divergence
  f^k &&\textrm{ in }{Q_T},\\
  u^k&=0&&\textrm{ on } \partial \Omega \times (0,T),
  \\
  u^k(0)&=0&&\textrm{ in } \Omega.
  \end{aligned}
\end{align}
In addition, we know that $u^k$ belongs to the natural Bochner spaces
\begin{equation}\label{eq:approxC}
\begin{aligned}
u^k &\in L^2([0,T],W^{1,2}_0(\Omega; \mathbb{R}^N)) \cap C([0,T], L^2(\Omega; \mathbb{R}^N)) \\
\partial_t u^k &\in L^2([0,T],(W^{1,2}_0)^*(\Omega; \mathbb{R}^N))
\end{aligned}
\end{equation}
and fulfil for every $\varphi \in \mathcal{C}^1_0((-\infty,T)\times \Omega; \mathbb{R}^N)$
\begin{equation}
\label{aprox2}
\int_{Q_T} \left( -u^k\cdot \partial_t\varphi + A(z,\nabla u^k) \cdot \nabla \varphi \right) \dz = \int_{Q_T} f^k
  \cdot \nabla \varphi \dz.
\end{equation}
Our goal is to let $k\to \infty$ in \eqref{aprox2} and to show that there exists a limit $u$ which satisfies \eqref{eq:WL}.

\subsection{Uniform estimates}
We start with the estimates that are independent of $k$. To do so, we compare \eqref{eq:approxC} with the proper linear system.
Naturally,  $u^k$ also solves  the following linear system
\begin{align}
\label{eq:aproxL}
\begin{aligned}
  \partial_tu^k-\divergence \left( \tilde A (z) \nabla u^k \right) &= -\divergence \left(
  f^k  + \left( \tilde A (z) \nabla u^k - A(z, \nabla u^k) \right) \right)&&\textrm{ in }{Q_T},\\
  u^k&=0&&\textrm{ on } \partial \Omega \times (0,T),
  \\
  u^k(0,\cdot)&=0&&\textrm{ in } \Omega.
  \end{aligned}
\end{align}
Defining the auxiliary weight $\omega_0^n:=\min\{n,\omega_0\}$ (which is bounded), we see that $A_2(\omega_0^n)\le 1+A_2(\omega_0)$ and thanks to \eqref{eq:approxC}, we can use
Theorem \ref{thm:lin} to observe that
 \begin{equation}\label{approx:keyest}
 \norm{\nabla u^k}_{L_{\omega_0^n}^{2}({Q_T})} \leq C(\tilde A,\Omega, A_2(\omega_0)) \left( \norm{f^k}_{L_{\omega_0^n}^{2}({Q_T})} + \norm{ \tilde A (z) \nabla u^k - A(z, \nabla u^k) }_{L_{\omega_0^n}^{2}({Q_T})}  \right).
\end{equation}
Due to the definition of $\omega_0^n$, we know that the right hand side is finite and thanks to Assumption~\ref{ass:A}, we have
\[
\begin{aligned}
| \tilde A (z) {Q} - A(z; {Q})  | &\le \frac{| \tilde A (z) {Q} - A(z; {Q})  |}{|Q|} |Q| 1_{\{|{Q}|\ge m \}} + \left(| \tilde A (z) {Q}| + |A(z; {Q})  | \right)  1_{\{|{Q}|< m \}}   \\
&\le \epsilon (m) |Q|  + C_\epsilon(m)
\end{aligned}
\]
with $\lim_{m \to 0 }\epsilon (m) =0$.
Applying this relation to $Q= \nabla u^k$ and combining it with \eqref{cfn}, we see that
 \begin{equation}\label{approx:keyest2}
 \norm{\nabla u^k}_{L_{\omega_0^n}^{2}({Q_T})} \leq C(\tilde A,\Omega, A_2(\omega_0)) \left( \norm{f}_{L_{\omega_0^n}^{2}({Q_T})} +  \epsilon (m)  \norm{\nabla u^k}_{L_{\omega_0^n}^{2}({Q_T})}   + C_\epsilon(m) \right).
\end{equation}
Hence, we choose $m$ so large that $\varepsilon C(\tilde A,\Omega, A_2(\omega_0))<\frac12$ and conclude
\begin{align*}
 \begin{aligned}
 \norm{\nabla u^k}_{L_{\omega_0^n}^{2}({Q_T})} \leq C (1+ \norm{f}_{L_{\omega_0^n}^{2}({Q_T})}  )\leq C (1+ \norm{f}_{L_{\omega_0}^{2}({Q_T})}  ).
 \end{aligned}
\end{align*}
Finally, we let $n\to \infty$ on the left hand side to get
\begin{align}
\label{aprior}
 \begin{aligned}
 \norm{\nabla u^k}_{L_{\omega_0}^{2}({Q_T})} \leq C (1+ \norm{f}_{L_{\omega_0}^{2}({Q_T})}  ),
 \end{aligned}
\end{align}
with $C$ depending merely on parameters of Assumption \ref{ass:A}, on $\Omega$ and $A_2(\omega_0)$. Analogously, we can obtain
\begin{align}
\label{apriorq}
 \begin{aligned}
 \norm{\nabla u^k}_{L_\omega^{q}({Q_T})} &\leq C (1+ \norm{f}_{L_\omega^{q}({Q_T})}  ),
 \\
 \norm{\nabla u^k}_{L^{q_0}({Q_T})} &\leq C (1+ \norm{f}_{L^{q_0}({Q_T})}  )
 \end{aligned}
\end{align}
with $C$ depending on parameters of Assumption~\ref{ass:A}, on $\Omega$, $q$ and $A_q(\omega)$.

\subsection{Weak limits}
Using the estimates \eqref{aprior} and \eqref{apriorq}, the reflexivity of the corresponding spaces and the growth given by Assumption \ref{ass:A},  we can pass to a subsequence (still denoted by $u^k$) such that
\begin{align}
 \label{conn-a}
 u^k &\rightharpoonup u &&\textrm{weakly in } L^{q_0}(0,T;W^{1,q_0}_0(\Omega; \setR^N)),
 \\
  \label{conn-b}
  \nabla u^k &\rightharpoonup \nabla u &&\textrm{weakly in }
  L^2_{\omega_0}\cap L_\omega^{q}\cap L^{q_0}({Q_T}; \setR^{n\times N}),\\
  A(x,\nabla u^k) &\rightharpoonup \overline{A} &&\textrm{weakly in }
  L^2_{\omega_0}\cap L_\omega^{q}\cap L^{q_0}({Q_T}; \setR^{n\times N})\label{con2}.
\end{align}
Next, using \eqref{conn-a}--\eqref{con2} and \eqref{cfn} in \eqref{aprox2} we obtain
\begin{equation}
\label{aprox3a}
\int_{Q_T} (-u\cdot\partial_t\phi + \overline{A} \cdot \nabla \phi ) \dz = \int_{Q_T}f \cdot \nabla \phi \dz
\end{equation}
for all $\phi \in \mathcal{C}^{0,1}_0((-\infty_0, T)\times \Omega))$. Moreover, the estimates \eqref{aprior}--\eqref{apriorq} remain valid also for $u$ due to the weak lower semicontinuity. Hence, $u$ satisfies \eqref{eq:opt}.

\subsection{Identification of the nonlinear limit}
There remains the most  difficult part, i.e. showing that
\begin{align}
\label{aim0}
  \overline{A}(z)&=A(z,\nabla u(z)) \qquad\textrm{ a.e. in  } Q_T.
\end{align}
Then, it follows from \eqref{aprox3a} that $u$ solves \eqref{eq:WL}.

We start the proof of \eqref{aim0} by showing that
\begin{align}\label{conn-c}
 u^k &\to u
 &&\textrm{strongly in }
  L^1({Q_T}; \setR^{N}).
\end{align}
Indeed, \eqref{conn-a} and \eqref{cfn} imply that $\partial_tu^k$ is bounded in 
$L^{q_0}(0,T;(W^{1,q'_0}_0(\Omega; \setR^N))^*)$, hence the Aubin-Lions argument implies \eqref{conn-c}.

Let us focus on \eqref{aim0}. For simplicity we denote $A^k:=A(\cdot, \nabla u^k)$ and set
$g_k:=(|A^k|^2 +\abs{\nabla u^k}^2+|f^k|^2 + \abs{\nabla u}^2 + |\overline{A}|^2)\omega_0$. Due to \eqref{conn-b}--\eqref{conn-c} and \eqref{cfn}, we see that the sequence $\{g^k\}_{k}$ is  bounded in $L^1(Q_T)$ and using the Chacon biting lemma, i.e., Lemma \ref{thm:blem}, we can find  a sequence of sets $E_j\subset Q_T$ such that  $\abs{Q_T\setminus E_j}\to 0$ as $j\to \infty$ and  such that for any $E_j$, one has that $g_k$ are equiintegrable in $E_j$. Hence,
 there exists a (non relabeled) subsequence such that
$$
 A^k\cdot \nabla (u^k-u) \omega_0 \weakto \xi_j \qquad \textrm{ weakly in } L^1 (E_j).
$$
The challenge is now to prove
\begin{align}
\label{aim1}
\int_{E_j}\xi_j \dz = 0.
\end{align}
Indeed, if we assume that \eqref{aim1} is satisfied, then it follows directly from \eqref{conn-b}--\eqref{conn-c} that
\[
\begin{aligned}
{\lim}_{k\to \infty}\int_{E_j}A^k\cdot \nabla u^k \omega_0\dz  &={\lim}_{k\to \infty}\int_{E_j}A^k\cdot \nabla (u^k-u) \omega_0\dz +{\lim}_{k\to \infty}\int_{E_j}A^k\cdot \nabla u \omega_0\dz\\
&= \int_{E_j}\overline{A}\cdot\nabla u \omega_0 \dz.
\end{aligned}
\]
Consequently, since $\omega_0$ is positive a.e. in $Q_T$, we can use the monotonicity and growth assumption on $A$, i.e., Assumption~\ref{ass:A}, to deduce that  for any $B \in L^2_{\omega_0} (Q_T)$
\[
0\leq \lim_{k\to \infty} \int_{E_j}(A^k-A(z, B))\cdot(\nabla u^k-B) \, \omega_0\dz =\int_{E_j} (\overline{A}-A(z, B))\cdot(\nabla u-B) \, \omega_0\dz.
\]
Therefore, taking $j \to \infty$, we obtain (note that the quantity is integrable thanks to \eqref{conn-b}--\eqref{conn-c} and we can use the Lebesgue dominated convergence theorem)
\[
0\leq \int_{Q_T} (\overline{A}-A(z, B))\cdot(\nabla u-B) \, \omega_0\dz  \quad (<\infty).
\]
Hence, the Minty trick allows to reconstruct the nonlinearity, i.e., \eqref{aim0} is established. For all the details we refer e.g. to \cite[pp. 4263--4264]{BulBurSch16}.

Thus, it remains to prove \eqref{aim1}. For brevity, we set $w^k:=u^k-u$ and $G^k:=A^k-\overline{A} -f^k +f$. Hence, it follows from \eqref{eq:aprox} and \eqref{aprox3a} that the sequence $\{w^k,G^k\}_{k}$ satisfies \eqref{eq:w}, which is one of the assumption of Theorem~\ref{thm:lipseq}. The second assumption \eqref{apkk} follows from \eqref{cfn} and \eqref{conn-b}--\eqref{con2}. Therefore, for any $\Lambda>0$ we can find a sequence $\{w^k_{\Lambda}\}_{k}$ fulfilling \ref{LS1}--\ref{LS4}. First, it directly follows from \ref{LS1}, \ref{LS2} and \eqref{conn-c} that
\begin{equation}
w^k_{\Lambda} \to 0 \quad \textrm{strongly in }L^p(Q_T;\mathbb{R}^N),\label{wkstr}
\end{equation}
for any $p\in [1,\infty]$. Hence, combining this result with \ref{LS3}, we obtain that for any $\eta \in \mathcal{C}^{0,1}_0(Q_T)$
\begin{equation}\label{tototo}
\begin{aligned}
\lim_{k\to \infty} \int_{Q_T} G^k \cdot \nabla w^k_{\Lambda} \eta \dz &=\lim_{k\to \infty} \int_{Q_T} G^k \cdot \nabla (w^k_{\Lambda} \eta) \dz \\
&=-\lim_{k\to \infty} \int_{Q_T} \partial_t w^k_{\Lambda} \cdot (w^k_{\Lambda}-w)\eta \dz.
\end{aligned}
\end{equation}
Next, due to the bound \ref{LS1} and \eqref{wkstr} we also have (for a subsequence)
\begin{equation}
\nabla w^k_{\Lambda} \rightharpoonup^* 0 \quad \textrm{weakly$^*$ in }L^{\infty}(Q_T;\mathbb{R}^{n\times N}).\label{wkstr2}
\end{equation}
Hence, using \eqref{wkstr}, \eqref{cfn} and \eqref{con2} and combining the result with \eqref{tototo}, we have
\begin{equation}\label{tototo2}
\begin{aligned}
\lim_{k\to \infty} \int_{Q_T} A^k \cdot \nabla w^k_{\Lambda} \eta \dz &=-\lim_{k\to \infty} \int_{Q_T} \partial_t w^k_{\Lambda} \cdot (w^k_{\Lambda}-w)\eta \dz.
\end{aligned}
\end{equation}
Next, using the density of smooth functions in $L^{q_0'}$, the uniform (independent of $k$) bounds \ref{LS1} and \eqref{con2}, we see that \eqref{tototo2} holds also for all $\eta \in L^{q_0'}(Q_T)$. Therefore, setting for an arbitrary $\ell>0$ a bounded $\eta:=\chi_{E_j}\min(\omega_0,\ell)=\chi_{E_j}\omega_0^{\ell}$, we arrive at
\begin{equation*}
\begin{aligned}
\lim_{k\to \infty} \left|\int_{E_j} A^k \cdot \nabla w^k_{\Lambda} \omega_0^{\ell} \dz\right| &=\lim_{k\to \infty} \left|\int_{E_{j}} \partial_t w^k_{\Lambda} \cdot (w^k_{\Lambda}-w)\omega_0^{\ell} \dz\right|.
\end{aligned}
\end{equation*}
Consequently, with the help of \ref{LS2} and the fact that $\omega_0^{\ell}\le \omega_0$, we can estimate the right hand side as
\begin{equation}\label{tototo3}
\begin{aligned}
\lim_{k\to \infty} \left|\int_{E_j} A^k \cdot \nabla w^k_{\Lambda} \omega_0^{\ell} \dz\right| &\le \lim_{k\to \infty} \int_{Q_T} |\partial_t w^k_{\Lambda} \cdot (w^k_{\Lambda}-w)|\omega_0 \dz \le \frac{C}{\sqrt{\Lambda}}.
\end{aligned}
\end{equation}
Finally, with the help of the H\"{o}lder inequality, the estimate \ref{LS2}, the definition \ref{LS4}, the definition of $g_k$, the triangle inequality and the relation \eqref{tototo3}, we have
\begin{align*}
\lim_{k\to \infty} \left|\int_{E_j}\xi_j \dz\right| &= \lim_{k\to \infty} \left|\int_{E_j} A^k \cdot \nabla w^k\omega_0\dz\right|\\
&\le \lim_{k\to \infty} \left|\int_{E_j} A^k \cdot \nabla w^k\omega_0^{\ell} \dz\right|+\lim_{k\to \infty} \int_{E_j \cap \{\omega_0\ge \ell\}} g_k\dz\\
&\le \lim_{k\to \infty} \left|\int_{E_j} A^k \cdot \nabla w^k_{\Lambda}\omega_0^{\ell} \dz\right|+\lim_{k\to \infty} \left|\int_{E_j\cap \mathcal{O}^k_{\Lambda}} A^k \cdot \nabla (w^k-w^k_{\Lambda})\omega_0^{\ell} \dz\right|\\
 &\quad+ \lim_{k\to \infty} \int_{E_j \cap \{\omega_0\ge \ell\}} g_k\dz\\
 &\le \frac{C}{\sqrt{\Lambda}} +  \lim_{k\to \infty} C\int_{E_j \cap (\{\omega_0\ge \ell\}\cup \mathcal{O}^k_{\Lambda})} (1+g_k)\dz
\end{align*}
Hence, since $\omega_0 \in L^1$ and we have \ref{LS4}, we know that
$$
|(\{\omega_0\ge \ell\}\cup \mathcal{O}^k_{\Lambda})| \le \frac{C}{\Lambda} + \frac{C}{\ell}
$$
and using the equiintegrability of the sequence $\{g_k\}_k$ on the set $E_j$ we find that
\begin{align*}
\lim_{k\to \infty} \left|\int_{E_j}\xi_j \dz\right| &\le \limsup_{\Lambda \to \infty} \limsup_{\ell \to \infty} \left(\frac{C}{\sqrt{\Lambda}} +  \lim_{k\to \infty} C \int_{E_j \cap (\{\omega_0\ge \ell\}\cup \mathcal{O}^k_{\Lambda})} (1+g_k) \dz\right)=0.
\end{align*}
Thus, \eqref{aim1} is proved and therefore \eqref{aim0} holds. Hence, $u$ is a solution.

\subsection{Uniqueness}
This section is heavily inspired by its `elliptic' counterpart \cite{BulDieSch15}.
Let two solutions $u_1$ and $u_2$ with data $(g_1,F_1,f)$ and $(g_2,F_2,f)$ satisfy the assumptions of Theorem~\ref{th}. Then defining $w:=u_1-u_2$ and using the compatibility condition \eqref{uniq:nonlin}, we get that $w\in L^s(0,T; W^{1,s}_0(\Omega;\mathbb{R}^N))$ with some $s>1$ solves
\begin{equation}\label{eq:WLpro}
\begin{split}
 &\int_{{Q_T}} \left[ - w (z) \cdot \partial_t \phi (z) +  \tilde{A}(z)\nabla w \cdot \nabla \phi (z)\right] \dz =\\
 & \int_{{Q_T}} \left[ \left(\tilde{A}(z)\nabla w - (A(z; \nabla u_1 (z))-A(z; \nabla u_2 (z))) \right) \cdot \nabla \phi (z)\right] \dz
\end{split}
\end{equation}
for an arbitrary $\varphi \in \mathcal{C}^{\infty}((-\infty,T)\times \Omega; \mathbb{R}^N)$. Next, we set
$$
\omega^k:=\min\{1, k[(M(\nabla u_1))^{s-2}+(M(\nabla u_1))^{s-2}]\}
$$
and using Lemma~\ref{cor:dual} we see that $\omega^k \in \Acal_{2}$ and
$$
A_2(\omega^k) \le 1 + A_2(k[(M(\nabla u_1))^{s-2}+(M(\nabla u_1))^{s-2}])=  1 + A_2((M(\nabla u_1))^{s-2}+(M(\nabla u_1))^{s-2})\le C.
$$
With such a weight, we can use the growth assumption on $A$ and $\tilde{A}$, i.e. Assumption \ref{ass:A}, and we obtain that
$$
F:= \tilde{A}(z)\nabla w - (A(z; \nabla u_1 (z))-A(z; \nabla u_2 (z))) \in L^2_{\omega^k}(Q_T;\mathbb{R}^{n\times N}).
$$

Hence, we can use the weighted linear theory, i.e. Theorem~\ref{thm:lin} to conclude
\begin{equation}
\label{un3}
\int_{Q_T}| \nabla w |^2 \omega^{{k}} \dz \le C \int_{Q_T} \left| \tilde{A}(z)\nabla w - (A(z; \nabla u_1 (z))-A(z; \nabla u_2 (z)))\right|^2 \omega^k \dz,
\end{equation}
where the constant $C$ independent of $k$.

Next, it follows from  Assumptions \ref{ass:A} and  \ref{ass:B} that (for details we refer to \cite{BulDieSch15})  for every $\delta>0$ there exists $C$ such that for all $z\in Q_T$ and all $Q, P \in \mathbb{R}^{n\times N}$ it holds
\begin{equation}\label{algebra}
\left| \tilde A (z) (Q -  P)  -  \left(A(z, Q) - A(z, P) \right) \right| \le \delta |Q-P| + C(\delta),
\end{equation}
The inequality \eqref{algebra} used in \eqref{un3} gives for any $\delta>0$
\begin{equation}
\label{un4}
\begin{split}
\int_{Q_T}| \nabla w |^2 \omega^{{k}} \dz \le C  \delta \int_{Q_T}| \nabla w |^2 \omega^{{k}}  \dz + C(\delta)\int_{Q_T}\omega^k\dz.
\end{split}
\end{equation}
Thus, setting $\delta$ sufficiently small  yields
\begin{equation}
\label{un5}
\begin{split}
\int_{Q_T}| \nabla w |^2  \omega^k \dz \le C\int_{Q_T} \omega^k \le C,
\end{split}
\end{equation}
where the last inequality follows from the fact that $Q_T$ is bounded and $\omega^k\le 1$. Hence, letting $k\to \infty$ in \eqref{un5}, together with $\omega^k \nearrow 1$ and the monotone convergence theorem implies $\nabla w = \nabla(u_1 - u_2) \in L^2 (Q_T; \mathbb{R}^{n\times N})$. Using \eqref{algebra}, we see that also
$$
(A(\cdot; \nabla u_1)-A(\cdot; \nabla u_2))\in L^2(Q_T; \mathbb{R}^{n\times N}).
$$
Consequently, going back to \eqref{eq:WL}, it also follows that
$$
\partial_t(u_1-u_2)=\partial_t w\in L^2(0,T; (W^{1,2}_0(\Omega;\mathbb{R}^N))^*)
$$
and due to the standard interpolation theorem and thanks to \eqref{eq:WL}, we have that $u_1-u_2 \in \mathcal{C}(0,T; L^2(\Omega;\mathbb{R}^N)$, $u_1(0)-u_2(0)=0$ and
$$
\int_0^T \left\langle \partial_t (u_1-u_2), \varphi\right\rangle + \int_{\Omega} (A(\cdot; \nabla u_1)-A(\cdot; \nabla u_2)) \cdot \nabla \varphi\dz \dt =0
$$
for all $\varphi\in L^2(0,T; W^{1,2}_0(\Omega;\mathbb{R}^N))$. Consequently, we can choose $\varphi:= (u_1-u_2)$ and then the monotonicity of $A$ with a Gronwall argument that $u_1=u_2$ a.e. in $Q_T$ (naturally, unlike in the steady case, we do not need strict monotonicity here thanks to Gronwall).

\subsection{Inhomogenous case}
Here we just add some remarks on the proof for the inhomogeneous case. Let $(g,F,f)$ be given data and we look for a solution fulfilling \eqref{eq:WL}. Then defining a new unknown $v:=u-g$, solving inhomogenous case is equivalent to finding $v\in L^q(0,T; W^{1,q}_0(\Omega;\mathbb{R}^{N}))$ that satisfies
\begin{equation}\label{eq:WLnew}
 \int_{{Q_T}} \left[ (- v(z)) \cdot \partial_t \phi (z) + A(z; \nabla (v-g) (z)) \cdot \nabla \phi (z)\right] \dz \\
  =\int_{Q_T}  (f (z)+F(z)) \cdot \nabla \phi (z) \dz\,.
\end{equation}
Let us thence define an approximative problem
\begin{equation}\label{eq:WLnewap}
\begin{split}
& \int_{{Q_T}} \left[ (- v^k(z)) \cdot \partial_t \phi (z) + A(z; \nabla v^k-(\nabla g)^k) (z)) \cdot \nabla \phi (z)\right] \dz \\
&\quad  =\int_{Q_T}  (f^k (z)+F^k(z)) \cdot \nabla \phi (z) \dz\,,
 \end{split}
\end{equation}
where
$$
(\nabla g)^k:=\nabla g \chi_{\{|\nabla g|\le k\}}, \quad f^k:=f\chi_{\{|f|\le k\}}, \quad F^k:=F\chi_{\{|F|\le k\}}.
$$
The problem \eqref{eq:WLnewap} has for each $k$ a unique solution $v^k\in L^2(0,T; W^{1,2}_0(\Omega;\mathbb{R}^N))$. Hence the a~priori estimates (independent of $k$) are obtained by comparing the problem \eqref{eq:WLnewap} with the linear problem as follows
\begin{equation}\label{eq:WLnewaplin}
\begin{split}
& \int_{{Q_T}} \left[ (- v^k(z)) \cdot \partial_t \phi (z) + \tilde{A}(z)\nabla v^k) \cdot \nabla \phi (z)\right] \dz \\
&\quad  =\int_{Q_T}  \left(f^k (z)+F^k(z)+\tilde{A}(z)\nabla v^k -A(z; \nabla v^k-(\nabla g)^k) (z) \right) \cdot \nabla \phi (z) \dz\,.
 \end{split}
\end{equation}
Using Assumption~\ref{ass:A}, we have pointwisely that for all $\varepsilon>0$ there exists $C(\varepsilon)$ such that
$$
\begin{aligned}
\left|\tilde{A}(z)\nabla v^k -A(z; \nabla v^k-(\nabla g)^k) (z)\right|&\le C|\nabla g| + \left|\tilde{A}(z)(\nabla v^k-(\nabla g)^k) -A(z; \nabla v^k-(\nabla g)^k) (z)\right|\\
&\le C|\nabla g| + \varepsilon |\nabla v^k-(\nabla g)^k| + C(\varepsilon)\le \varepsilon|\nabla v^k| + C(\varepsilon) |\nabla g|.
\end{aligned}
$$
Hence, we can proceed exactly in the same way as in the homogeneous case and obtain estimates on $v^k$ depending now on $F$, $g$ and $f$. The limit procedure follows almost step by step the limit procedure of homogeneous case to obtain a solution $v$ of \eqref{eq:WLnew}. Finally, defining $u:=v+g$, it follows from \eqref{eq:WLnew} that it is a weak solution. The proof is complete. 
\begin{proof}[Proof of Corollary~\ref{CMB}]
Let us solve auxiliary scalar problems for each scalar component $\mu^i$, $i = 1, \dots N$ of $\mu$
\[
\begin{aligned}
  \partial_th^i-\Delta h^i &= \mu^i &&\textrm{ in } Q_T,\\
  h^i&=0&&\textrm{ on } \partial Q_T
  \end{aligned}
\]
that have a unique solution in $h=(h^1,...,h^N)\in L^s(0,T; W^{1,s}(\Omega;\RN))$ for $s< \frac{n}{n-1}$, cf.~\cite{Bocatall97,Casas97}.
Hence, the function $u$ exists if and only if $w=u-h$ satisfies
\[
\partial_t w-\divergence( A(\cdot,\nabla w+\nabla h))=-\divergence(\nabla h)\text{ in }\Omega, \quad w=0\text{ on }\partial \Omega.
\]
We now follow the existence proof of Theorem~\ref{thweak}. We define $H_k=
\min\{\abs{\nabla h},k\}\frac{\nabla h}{\abs{\nabla h}}$. Then we first solve
\[
\partial_t w_k-\divergence( A(\cdot,\nabla w_k+H_k)=-\divergence(H_k)\text{ in }\Omega, \quad w_k=0\text{ on }\partial \Omega
\]
This operators satisfy the necessary bounds such that we have the existence of a sequence $w_k$ that satisfies uniform a-priori bounds in weighted spaces. By imitating the argument in Theorem~\ref{thweak} we gain the existence of a solution $w$ and so the existence of $u:=w-h$. Next the stability and uniqueness of $w$ follows in precisely the same way as in the proof of Theorem~\ref{thweak}. 
 Finally, the uniqueness for $u$ follows by the uniqueness of $h$ and $w$. 
\end{proof}

\end{document}